\documentclass[12pt, reqno]{amsart}

\usepackage{amssymb}
\usepackage{latexsym}
\usepackage{amsmath}
\usepackage{amsthm}
\usepackage{amsfonts}
\usepackage{color}
\usepackage[final]{hyperref}
\usepackage{pdfsync}
\usepackage[usenames,dvipsnames]{pstricks}
\usepackage{epsfig}
\usepackage{pst-grad} 
\usepackage{pst-plot} 
\usepackage{accents}
\newcommand{\ubar}[1]{\underaccent{\bar}{#1}}
\usepackage{dsfont}
\usepackage{amsopn,esint}
\usepackage[margin=1in]{geometry}
\usepackage{upref}
\usepackage[mathscr]{eucal}

\newcommand{\R}{\mathbb{R}}
\newcommand{\N}{\mathbb{N}}

\newcommand{\Ls}{(-\Delta)_\Omega^s}

\newcommand{\OO}{\mathcal O}

\theoremstyle{plain}
\newtheorem{defi}{Definition}[section]
\newtheorem{prop}[defi]{Proposition}
\newtheorem{teo}[defi]{Theorem}

\newtheorem{lema}[defi]{Lemma}

\newtheorem{remark}[defi]{Remark}

\theoremstyle{definition}

\theoremstyle{remark}

\numberwithin{equation}{section}

\begin{document}


\title[]{Ergodic pairs for fractional Hamilton-Jacobi equations on bounded domains: large solutions.}

%
\author[]{Alexander Quaas}
\address{
Alexander Quaas: Departamento de Matem\'atica, Universidad T\'ecnica Federico Santa Mar\'ia \\
Casilla: v-110, Avda. Espa\~na 1680, Valpara\'iso, Chile}
\email{alexander.quaas@usm.cl}

\author[]{Erwin Topp}
\address{
Erwin Topp:
Instituto de Matem\'aticas, Universidade Federal do Rio de Janeiro, Rio de Janeiro - RJ, 21941-909, Brazil. {\tt etopp@im.ufrj.br}
}



\date{\today}

\begin{abstract}
In this article, we study the ergodic problem associated to viscous Hamilton-Jacobi equation where the diffusion is governed by the censored fractional Laplacian, a nonlocal elliptic operator restricted to a bounded domain $\Omega \subset \R^N$. We restrict ourselves to the case in which the nonlinear gradient term has a scaling less or equal than the fractional order of the diffusion. In similarity to its second-order counterpart, we provide existence of ergodic pairs involving solutions that blow-up on $\partial \Omega$. We use the celebrated vanishing discount method, where the analysis of the approximated solutions have its own interest, leading to qualitative properties for the ergodic problem such as precise blow-up rates for the solution and characterization of the ergodic constant. The main difficulties arise from the state-dependency of the operator, from which the arguments of the local case based on well-known invariance properties of the Laplacian are not longer at disposal.
\end{abstract}

\maketitle

\section{Introduction.}

Let $\Omega \subset \R^N$ be bounded domain with $C^2$ boundary. For $s \in (0, 1)$, we denote $\Ls$ the \textsl{censored (or regional) fractional Laplacian of order $2s$}. It is given by the formula
$$
\Ls u(x) = -C_{N,s}  \mathrm{P.V.} \int_{\Omega} \frac{u(z) - u(x)}{| x - z|^{N + 2s}}dz,
$$
whenever the integral makes sense for $u: \Omega \to \R$ measurable and $x \in \Omega$. The integral is understood in the Cauchy Principal Value sense, and $C_{N, s} > 0$ is a well-known normalizing constant, see for instance~\cite{reg1}. 

This is an integro-differential operator that arise in different mathematical contexts. To name a few, we can find it as the infinitesimal generator of stochastic stable-like processes whose trajectories are confined to $\Omega$~(\cite{BBC}). From the point of view of the Calculus of Variations, it arises in the Euler-Lagrange formulation for minimizers of the Sobolev energy in $H^s(\Omega)$~(\cite{reg1}). In a more PDE setting, it is a  compatible fractional operator to deal with Neumann-type problems, see for instance~\cite{bcgj}.


We are going to consider it in the context of Hamilton-Jacobi  equations (\textsl{HJ equations} for short). Our setting assumes that $s \in (1/2, 1)$, $s + 1/2 < m \leq 2s$, and $f \in C(\Omega)$ is bounded from below. It is our interest the analysis on the existence and qualitative properties of a \textsl{pair} $(u, c) \in C(\Omega) \times \R$ solving the \textsl{ergodic problem}
\begin{equation}\label{eq-erg-frac}\tag{${\bf E}$}
\Ls u + |Du|^m = f - c \quad \mbox{in} \ \Omega,
\end{equation}
subject to the boundary blow-up condition, 
\begin{equation}\label{blow-up}\tag{${\bf BU}$}
u(x) \nearrow +\infty \quad \mbox{as} \ x \to \partial \Omega,
\end{equation}
adopting the notion of viscosity solution to address this problem.

Notice that both~\eqref{eq-erg-frac} and~\eqref{blow-up} remain the same if we add a constant to $u$, from which the pair $(u, c)$ may be cast as the solution to an \textsl{additive eigenvalue problem}. 
This type of eigenvalue problem has been the subject of interest of the (degenerate) elliptic PDE community since the work by Lions, Papanicoloau and Varadhan in~\cite{LPV} for first-order problems in periodic ambient space, and subsequently extended to second-order and nonlocal setting by various authors since then.

To treat~\eqref{eq-erg-frac}-\eqref{blow-up},
we follow the approach by Lasry and Lions~\cite{LL} for viscous HJ equations posed on bounded domains. The analysis starts with the \textsl{discounted problem}, which considers a parameter $\lambda > 0$, and address the solvability of the equation
\begin{equation}\label{eqLL}
\lambda u - \Delta u + |Du|^m = f \quad \mbox{in} \ \Omega,
\end{equation}
in the sub-quadratic case $1 < m \leq 2$. Here, $\lambda > 0$ is usually referred to as the \textsl{discount factor}, and its role into~\eqref{eqLL} makes it a \textsl{proper equation}, a general condition that allows the use of viscosity comparison principles in degenerate elliptic equations. In fact, by means of such comparison properties together with the method of sub and supersolutions, it is proven in~\cite{LL} the existence of a unique solution $u = u_\lambda$ to~\eqref{eqLL}-\eqref{blow-up} for each $\lambda > 0$. By its boundary behavior~\eqref{blow-up}, we say that $u$ is a \textsl{large solution} for the equation under study.

The method yields a precise blow-up profile of these large solutions in terms of the distance to the boundary function $d(x) := \mathrm{dist}(x, \partial \Omega)$.
From here, it is possible to address the ergodic problem through the so-called \textsl{vanishing discount problem}: after a suitable normalization of $u_\lambda$ and thanks to interior elliptic estimates, we can extract a converging subsequence of the family $\{ u_\lambda \}_\lambda$ as $\lambda \searrow 0$, concluding the existence of a pair $(u,c)$ solving the equation
\begin{equation}\label{eqLL-erg}
- \Delta u + |Du|^m  = f - c \quad \mbox{in} \ \Omega,
\end{equation}
with $u$ satisfying~\eqref{blow-up}. Moreover, the ergodic constant $c$ is unique, and the solution $u$ is unique up to an additive constant, consequence of standard elliptic tools such as interior regularity estimates and Strong Maximum/Comparison Principles. This type of eigenvalue problems is employed into the study the large time behavior of parabolic HJ equations~\cite{B3}. The ergodic pair has an interpretation in terms of an stochactic optimal control with state-constraint, see~\cite{FGMP13, LL}. Finally, 
it is worth to mention that the superquadratic case $m > 2$ is also studied in~\cite{LL}, and in this case solutions for the ergodic problem remain bounded.




\smallskip

In the nonlocal setting, the authors of this manuscript in collaboration with G. D\'avila~\cite{DQTe} studied the fractional HJ equation
\begin{equation}\label{eq-frac0}
\lambda u + (- \Delta)^s u + |Du|^m = f  \quad \mbox{in} \ \Omega,
\end{equation}
where $(-\Delta)^s$ denotes the (full) fractional Laplacian of order $2s$, given by 
$$
(-\Delta)^s u(x) = -C_{N,s} \int_{\R^N} \frac{u(z) - u(x)}{| x - z|^{N + 2s}}dz.
$$

Notice that for this operator the integration takes into account the information of $u$ in the whole of $\R^N$, from which~\eqref{eq-frac0} must be complemented with an \textsl{exterior condition} in $\Omega^c$. For instance, given $\varphi: \Omega^c \to \R$ adequate, we can impose
\begin{equation}\label{exterior}
u = \varphi \quad \mbox{in} \ \Omega^c,
\end{equation}
in addition to~\eqref{blow-up}. There are multiple solutions to~\eqref{eq-frac0}-\eqref{exterior} satisfying~\eqref{blow-up}, consequence of the integration on $\Omega^c$ which implies nontrivial blow-up traces on $\partial \Omega$, see~\cite{grubb1, Ab, CFQ}.

Moreover, no ergodic behavior is observed in this type of \textsl{non censored} problems. A formal way to see this is in the case~\eqref{exterior} holds, can be explained by the identity
\begin{align*}
(-\Delta)^s u(x) = \Ls u(x) + \lambda_s(x) u(x), \quad \mbox{with} \quad \lambda_s(x) := C_{N,s}\int_{\Omega^c} |x - y|^{-(N + 2s)}dy,
\end{align*}
valid for all $x \in \Omega$, from which~\eqref{eq-frac0} can be formally written as
\begin{align*}
(\lambda + \lambda_s) u + \Ls u + |Du|^m = f + \tilde \varphi \quad \mbox{in} \ \Omega, \quad \mbox{with} \quad \tilde \varphi(x) := C_{N,s}\int_{\Omega^c} \frac{\varphi(y)}{|x - y|^{N + 2s}}dy,
\end{align*}
which is a proper equation even if $\lambda = 0$.

From the above discussion, the censored fractional Laplacian is a natural operator to address fractional ergodic problems on bounded domains. As we mentioned above, we start dealing with the discounted problem: for $\lambda > 0$, we consider
\begin{equation}\label{eq}\tag{${\bf P_\lambda}$}
\lambda u + \Ls u + |Du|^m  = f \quad \mbox{in} \ \Omega.
\end{equation}


%

%
%
%

We introduce some notation that will be used in the rest of the paper. For $1/2 + s < m \leq 2s$, we define $\gamma \in [0, 1)$ as
\begin{equation}\label{defgamma}
\gamma = \frac{2s - m}{m - 1},
\end{equation}
and denote
$$
g_\gamma(x) = \left \{ \begin{array}{cl} d(x)^{-\gamma} \quad & \mbox{if} \ \gamma > 0, \\
-\log(d(x)) \quad & \mbox{if} \ \gamma = 0, \end{array} \right .
$$

We say that a continuous function $v: \Omega \to \R$ is in the \textsl{$\gamma$-class} if
\begin{equation}\label{gamma-class}
-\infty < \liminf \limits_{d(x) \to 0^+} \frac{v(x)}{g_\gamma(x)} \leq \limsup \limits_{d(x) \to 0^+} \frac{v(x)}{g_\gamma(x)} < +\infty.
\end{equation}

Now we present our first main result in long extent.


\begin{teo}\label{teo1}
Let $s \in (1/2, 1)$ and $s + 1/2 < m \leq 2s$. Then, for each $\lambda > 0$ and $f \in C(\Omega)$ bounded from below such that
\begin{equation}\label{compor-f}
\lim_{d(x) \to 0^+} d(x)^{\gamma + 2s} f(x) = C_1.
\end{equation}
there exists a viscosity solution $u \in C^{1, \alpha}(\Omega)$ to problem~\eqref{eq}-\eqref{blow-up}. It is the minimal large solution to~\eqref{eq} in the sense that for every $v \in C(\Omega)$ viscosity solution to~\eqref{eq}-\eqref{blow-up}, then $u \leq v$ in $\Omega$. It is also the unique solution in the $\gamma$-class~\eqref{gamma-class}. Moreover, there exists $C_0 > 0$ depending on $s, N, m$ and $C_1$ such that 
\begin{align}\label{rate-frac}
 \lim \limits_{d(x) \to 0^+} \frac{u(x)}{g_\gamma(x)} = C_0.
\end{align}


Finally, there exists $C_* > 0$ depending on $\Omega, N, s, m$ and $C_1$ such that
\begin{align}\label{rate-frac1}
|D u(x)| \leq C_* \left \{ \begin{array}{ll} d(x)^{-\gamma - 1} \quad & \mbox{if} \ 1/2 + s < m < 2s, \\
d(x)^{- 1} \quad & \mbox{if} \ m = 2s, \end{array} \right .
\end{align}
for all $x \in \Omega$ close to $\partial \Omega$.
\end{teo}

The constant $C_0$ in~\eqref{rate-frac} is given by the unique positive root of the equations
\begin{equation}\label{defc0}
\left \{ \begin{array}{rll} -C_0 c_\gamma + \gamma^m C_0^{m}-C_1 & = 0 & \quad \mbox{if $m < 2s$}, \\
-C_0 c +  C_0^{2s}-C_1 & = 0 & \quad \mbox{if $m = 2s$}. \end{array} \right .
\end{equation}
where $c$ is the constant in Lemma~\ref{lemabarrera0} (which depends on $\gamma, N$ and $s$). This \textsl{main order} blow-up profile resembles the one of~\cite{LL}.




Some comments of the result are in order. The main contrast in relation to the local case is the partial uniqueness we presented here, only restricted to the $\gamma$-class. It is tempting to look at the multiplicity result for the ``full Laplacian problem", but $\Ls$ lacks exactly on the device promoting multiple solutions for~\eqref{eq-frac0}. See also Chen and Hajaiej~\cite{CH} where the case of semilinear censored fractional equations is addressed. The difficulty comes from the fact that operator $\Ls$ is not translation invariant. At this respect, it is often to see arguments based on the analysis on subdomains $\Omega' \subset \Omega$, which leads to the construction of \textsl{maximal solutions} for the problem on $\Omega$, concluding estimates from above for \textsl{every large solution of the equation}. However, $\Ls$ and $(-\Delta)_{\Omega'}^s$ are practically unrelated, even more if we deal with large solutions. By the structure of the operator and the method used here, the constraint $s + 1/2 < m$ seems to be necessary to have $g_\gamma \in L^1(\Omega)$ and build up suitable barriers.


\smallskip

Our second main result states the solvability of the ergodic problem.
\begin{teo}\label{teo2}
Let $s \in (1/2, 1)$, $s + 1/2 < m \leq 2s$, and $f \in C(\Omega)$ bounded below and locally H\"older continuous in $\Omega$ such that \eqref{compor-f} holds

Then, there exists $(u, c) \in C^{2s + \alpha}(\Omega) \times \R$ solving~\eqref{eq-erg-frac}-\eqref{blow-up}.

The constant $c$ meets the characterization
$$
c = \inf \{ \rho: \mbox{$\exists \ v \in C(\Omega)$ satisfying~\eqref{blow-up} and  $\Ls v + |Dv|^m \geq f-\rho$ in $\Omega$} \},
$$
and the solution $u$ satisfies the asymptotic boundary profiles~\eqref{rate-frac}-\eqref{rate-frac1}.

\smallskip

We also have:
\begin{itemize}
\item {\bf Minimality:} If $(v, c') \in C(\Omega) \times \R$ is any eigenpair solving~\eqref{eq-erg-frac}-\eqref{blow-up}, then there exists $M > 0$ such that $u - M \leq v$ in $\Omega$.

\smallskip

\item {\bf Partial uniqueness:} If $(v,c) \in C(\Omega) \times \R$ solve~\eqref{eq-erg-frac}-\eqref{blow-up}, then $v$ is in the $\gamma$-class~\eqref{gamma-class} and $v=u+C$ for some $C \in \R$. 

If $(v,c') \in C(\Omega) \times \R$ solve~\eqref{eq-erg-frac}-\eqref{blow-up} with $v$ in the $\gamma$-class~\eqref{gamma-class}, then $c' = c$ and $v=u+C$ for some $C \in \R$. 
\end{itemize}
\end{teo}

Let us immediately mention that the assumption $f$ locally H\"older continuous is not necessary, and we adopted it only for simplicity in order to avoid tedious viscosity 
arguments. We believe that this result can be readily extended for merely continuous $f$ by approximation, but we did not pursue in this direction.

As in Theorem~\ref{teo1}, we were unable to provide full uniqueness results for the ergodic pair. The intricate nonlocal structure prevents the application of standard methods to conclude, for example, monotonicity formulas for the ergodic constant $c$ with respect to $\Omega$, a desirable property for further applications. 

\subsection{Literature review and further discussion.} 
Large solutions for second-order ergodic problems on bounded domains have been addressed in more generality in~\cite{BDL19,BL21} where fully nonlinear and quasilinear versions of~\cite{LL} are studied. 

In our case, the elliptic operator has a nonlocal nature. We have chosen the censored fractional Laplacian since it is a basic model that arises in other contexts, such as  weak/variational problems related to $H^s(\Omega)$ energy~\cite{reg1, reg2, reg3, reg4}, and semilinear PDEs addressed through Green's function, see~\cite{CH}. In the viscosity setting, we can mention~\cite{bcgj, IN, G17, RT, BT16}. 

There are plenty of results for ergodic problems in the periodic setting, see for instance~\cite{BCCI14, BKLT15, BLT17} and its applications on ergodic large time behavior for related parabolic problems. Periodicity brings suitable compactness properties for solutions, which combined with rigidity results such as Strong Maximum Principle lead to uniqueness of the ergodic constant and uniqueness up to an additive constant for the solutions. In the non periodic setting, in~\cite{BT} the authors prove the existence of the ergodic pair for HJ equations posed in $\R^N$ with a coercive right-hand side $f$, which is the device bringing compactness properties to solve the problem. However, full uniqueness is still open in the general case. At this respect, different techniques used in the second-order setting such as Hopf-Cole transformation (\cite{B2}), Liouville theorems~\cite{BDL21}, or the use of probabilistic tools (\cite{A}) are not at full disposal for fractional equations, or they are difficult to implement. 

We believe that it is worth to find new tools to treat some of the open questions posed here, specially the ones linked  to uniqueness for the discounted problem and/or ergodic problem. The systematic development of auxiliary tools such as Liouville theorems on the half-space for censored operators, or higher order expansions of the boundary blow-up profile of solutions have their own interest, and seem to be a necessary step for a subsequent study of natural questions, for example, the ergodic large time behavior for parabolic HJ equations with censored fractional diffusion.

Finally, we mention that our methods can be readily  extended to simple extensions of the operator, for instance when the standard kernel of the fractional Laplacian is replaced by a function with the form $K(x - y)$ with $K$ symmetric, continuous in $\R^N \setminus \{ 0 \}$, positive $2s$-homogeneous and satisfying ellipticity conditions à la Caffarelli and Silvestre~\cite{CS09}. Moreover, we believe that our method are applicable to some classes of fully nonlinear version of~\eqref{eq-erg-frac} and~\eqref{eq}, see~\cite{DQThar}, and other censored operators as the ones presented in~\cite{IN, RT}. Other natural extensions such as non-symmetric nonlocal operators and degenerate operators are far more involved, as well as equations with a more general Hamiltonians acting on the gradient term. We did not pursue in this direction.

\smallskip

The paper is organized as follows: in Section~\ref{sectec} we provide basic tools such as barriers and local regularity with its proofs. In Section~\ref{secdiscount} we provide the proof of Theorem~\ref{teo1} and discuss certain non-existence results in the case $s \leq 1/2$. In Section~\ref{secerg}, we prove the existence for the ergodic problem, and provide the details of its qualitative properties in Section~\ref{secqual}, where we complete the proof of Theorem~\ref{teo2}.

\section{Technical lemmas}
\label{sectec}

We start with some notation to be assumed throughout the paper. 
For measurable $A \subseteq \Omega, x \in \Omega$ and $u: \Omega \to \R$ measurable, we denote
$$
(-\Delta)_\Omega^s[A] u(x) = C_{N,s} \mathrm{PV} \int_{\Omega \cap A} \frac{u(x) - u(y)}{|x - y|^{N + 2s}}dy.
$$

Recall that $d(x) = \mathrm{dist}(x, \partial \Omega)$. We use the notation
$$
\Omega_\delta = \{ x \in \Omega : d(x) > \delta\},
$$
for each $\delta > 0$. Since we assume $\partial \Omega$ is smooth, we have the existence of $\delta_1 > 0$ such that the distance function $d$ is smooth in $\Omega \setminus \Omega_{\delta_1}$.

\subsection{Barriers.}
We construct blow-up solutions using Perron's method. In order to construct suitable barriers, we require some lemmas presented in this section. 

We start with the following estimates for the censored fractional Laplacian on functions of the distance $d(x) = \mathrm{dist}(x, \partial \Omega)$. The main novelty here are the estimates for the logarithm, that will be used in the critical case $m = 2s$.

\begin{lema}\label{lemabarrera0}
	Let 
	 $s \in (0, 1)$. There exists $\delta_0 \in (0, \delta_1)$ depending on $N, s$ and $\Omega$ such that

\smallskip
\noindent
$(i)$ for each $\gamma \in (-1, 2s)$ and $s \in (1/2, 1)$, there exists $c = c(N, s, \gamma)  \in \R$ depending on $\gamma, N, s$ such that
\begin{equation*}
-\Ls d^\gamma(x) = d^{\gamma - 2s}(x) (c +  O(d^s(x))) \quad \mbox{for all} \ d(x) < \delta_0,
\end{equation*}
where the $O$ term depends on the data and the smoothness of $\partial \Omega$.

\smallskip

Moreover, for $N, s$ fixed, the map $\gamma \mapsto c(N, s, \gamma)$ is strictly convex and smooth, and we have that $c(0) = c(2s - 1) = 0$, and; therefore $c(\gamma) < 0$ if $\gamma \in (0, 2s - 1)$, and $c(\gamma) > 0$ if $\gamma \in (-1, 0) \cup (2s - 1, 2s)$.

\medskip

\noindent
$(ii)$ there exists $c = c(N, s) \in \R$ depending on $N, s$ such that
$$
-\Ls \log d(x) = d^{-2s} (c + O(d^s(x)|\log d(x)|)) \quad \mbox{for all} \ d(x) < \delta_0,
$$
where the $O$ term depends on the data and the smoothness of $\partial \Omega$. Moreover, $c(1/2) = 0$, $c(s) > 0$ if $s< 1/2$ and $c(s) < 0$ if $s> 1/2.$
\end{lema}

\begin{proof} The proof of $(i)$ follows by the estimates in Lemma B.1 in [2] where the one dimensional case is studied. Then we can use the argument as ~\cite[Proposition 3.1]{DQTe} to obtain a one dimensional reduction; so estimate follow from here. For the convexity and smoothness of $c$ can be proven as in  \cite[Proposition 3.1]{CFQ}, see also \cite[Lemma 2.2]{BDGQ}.

\medskip

For $(ii)$, we start describing the one dimensional case as in~\cite[Lemma A.1]{bcgj}. For $x > 0$, we have 
\begin{align*}
-(-\Delta)_{\R_+}^s \log(x) = & C_{1,s} x^{-(1 + 2s)} \mathrm{P.V.} \int_{0}^{+\infty} \frac{\log(y/x)}{|1 - y/x|^{1 + 2s}}dy \\
= & C_{1,s} x^{-2s} \mathrm{P.V.} \int_{0}^{+\infty} \frac{\log(t)}{|1 - t|^{1 + 2s}}dt \\
= & C_{1,s} x^{-2s} \mathrm{P.V.} \int_{-\infty}^{+\infty} \frac{z e^{z}}{|1 - e^z|^{1 + 2s}}dz \\
= & 2^{1 + 2s} C_{1,s} x^{-2s} \mathrm{P.V.} \int_{-\infty}^{+\infty} \frac{z e^{(1/2 - s)z}}{|\sinh(z/2)|^{1 + 2s}}dz,
\end{align*}
from which 
$$
-(-\Delta)_{\R_+}^s \log(x) = c_s x^{-2s}.
$$
with
$$
c_s = 2^{1 + 2s} C_{1,s} \mathrm{P.V.} \int_{-\infty}^{+\infty} \frac{z e^{(1/2 - s)z}}{|\sinh(z/2)|^{1 + 2s}}dz.
$$

Notice that $c_s > 0$ for $s < 1/2$, $c_{1/2} = 0$ and $c_s < 0$ for $s > 1/2$. 

\smallskip

Now we deal with the multidimensional case.
Let $\rho = d(x)$ and after rotation and translation, we can assume that $x = \rho e_N$ and the projection of $x$ onto $\partial \Omega$ is the origin.
For $\eta > 0$ small but independent of $\rho$, we consider the cylinder $Q_\eta = B_\eta' \times (-\eta, \eta) \subset \R^N$. Then, we write 
$$
\Ls \xi (x) = \Ls[Q_\eta] \xi(x) + \Ls[Q_\eta^c] \xi(x).
$$ 

If $\rho \leq \eta/4$, we have by the integrability of $x \mapsto \log (d(x))$ in $\Omega$ that
$$
\Ls[Q_\eta^c] \xi(x) = O(|\log \rho |). 
$$

For the integral on $Q_\eta$, we explicitly write
$$
-\Ls[Q_\eta] \xi (x) = \mathrm{PV} \int_{(\Omega - \rho e_N) \cap Q_\eta} \frac{\log(d(x + z)) -  \log \rho}{|z|^{N + 2s}}dz. 
$$

Now, using the above local chart $\psi$ around the origin, we have that
$$
d(x + z) \leq \rho + z_N - \psi(z'), \quad (z', z) \in Q_\eta.
$$

On the other hand, by Lemma 3.1 in~\cite{CFQ}, we have the existence of $C_\Omega > 0$ such that
$$
d(x + z) \geq (\rho + z_N - \psi(z'))(1 - C_\Omega |z'|^2), \quad (z', z) \in Q_\eta.
$$

Thus, denoting $\psi_\rho(y') = \rho^{-1} \psi(\rho y')$ for $y' \in B_{\eta/\rho}$, $\tilde \Omega = (\rho^{-1}\Omega - e_N) \cap Q_{\eta/\rho}$, and defining
$$
I = \int_{\tilde \Omega} \frac{\log(1 + y_N - \psi_\rho(y'))}{|y'|^{N + 2s}}dy,
$$
we have
$$
\rho^{-2s} I \leq -\Ls[Q_\eta] \xi (x) \leq \rho^{-2s} I + \int_{(\Omega - \rho e_N )\cap Q_\eta} \frac{\log(1 - C_\Omega |z'|^2)}{|z|^{N + 2s}}dz,
$$
from which, taking $\eta << 1$ in terms of $C_\Omega$, we may assume that $\log(1 - C_\Omega |z'|^2)$ is bounded when $z \in Q_\eta$, and since $\log(1 - C_\Omega |z'|^2) \leq C |z'|^2$ for some $C > 0$, we arrive at
$$
\rho^{-2s} I \leq \Ls[Q_\eta] \xi (x) \leq \rho^{-2s} I + C,
$$
 for some $C > 0$ not depending on $\rho$. Then, it remains to estimate $I$. For this, we consider the sets
$$
\tilde \Omega_1 = \tilde \Omega \cap \{ y : |y'| \leq 1 \}, \  \tilde \Omega_2 = \tilde \Omega \cap \{ y : 1 \leq |y'| \leq \eta/\sqrt{\rho} \}, \ \tilde \Omega_3 = \tilde \Omega \cap \{ y : \eta/\sqrt{\rho} < |y'|\},
$$ 
and $I_i$ the integral of the same integrand as $I$ on the corresponding sets, from which 
$$
I = I_1 + I_2 + I_3.
$$

For $I_1$, recalling that $|\psi_\rho(y')| \leq C \rho |y'|^2$ for some $C > 0$ just depending on $\Omega$, we have $|\psi_\rho(y')| \leq C \rho$ for all $|y'|\leq 1$, from which we may split the integral as
\begin{align*}
I_1 = & \int_{\tilde \Omega_1} \frac{\log(1 + y_N - \psi_\rho(y'))}{|y|^{N + 2s}}dy \\
= & \int_{B_1'} \int_{\psi_\rho(y') - 1}^{-1/2} \frac{\log(1 + y_N - \psi_\rho(y'))}{|y|^{N + 2s}}dy_N dy' + \int_{B_1'} \int_{1/2}^{\eta/\rho} \frac{\log(1 + y_N - \psi_\rho(y'))}{|y|^{N + 2s}}dy_N dy' \\
& + \int_{B_1'} \int_{-1/2}^{1/2} \frac{\log(1 + y_N - \psi_\rho(y'))}{|y|^{N + 2s}}dy_N dy' \\
= & \int_{B_1'} \int_{- 1}^{-1/2 - \psi_\rho(y')} \frac{\log(1 + t)}{(|t + \psi_\rho(y')|^2 + |y'|^2)^{(N + 2s)/2}}dt dy' \\
& + \int_{B_1'} \int_{1/2 - \psi_\rho(y')}^{\eta/\rho - \psi_\rho(y')} \frac{\log(1 + t)}{(|t + \psi_\rho(y')|^2 + |y'|^2)^{(N + 2s)/2}}dt dy' \\
& + \int_{B_1'} \int_{-1/2}^{1/2} \frac{\log(1 + y_N - \psi_\rho(y'))}{|y|^{N + 2s}}dy_N dy'
\end{align*}

For the first two integrals, we may assume that $\rho$ is small enough in order that $|\psi_\rho(y')| \leq 1/8$. Thus, in both integrals we have $|t| \geq 1/4$, from which
$$
(|t + \psi_\rho(y')|^2 + |y'|^2)^{-(N + 2s)/2} = (|t|^2 + |y'|^2)^{-(N + 2s)/2} (1 + O(\rho)),
$$
where $O(\rho)$ just depends on the smoothness of $\partial \Omega$. Moreover, observe that the function
$$
(y', t) \mapsto \frac{\log(1 + t)}{(|t|^2 + |y'|^2)^{(N + 2s)/2}}
$$
is uniformly bounded for $|t-1/2| \leq C\rho, |t + 1/2| \leq C\rho$ for all $\rho$ universally small (just depending on the smoothness of the boundary).

For the third integral, we use that for $|y_N| \leq 1/2$ we have
$$
\log(1 + y_N - \psi_\rho(y')) = \log(1 + y_N) (1 + O(\rho)),
$$
where $O(\rho)$ just depends on $\Omega$. Hence, we have
$$
I_1 = (1 + O(\rho)) \int_{B_1'} \int_{-1}^{\eta/\rho - \psi_\rho(y')} \frac{\log(1 + t)}{(|t|^2 + |y'|^2)^{(N + 2s)/2}} dt dy' + O(\rho),
$$
and using the integrability at infinity of the kernel, it is possible to conclude
\begin{equation}
I_1 = (1 + O(\rho)) \Big{(} \int_{B_1'}\int_{-1}^{+\infty} \frac{\log(1 + t)}{|(t,y')|^{N + 2s}} dt dy' + O(\rho^{2s} )\Big{)} + O(\rho)
\end{equation}

Now, for $I_2$, we see that in this region $|\psi_\rho(y')| \leq C \rho |y'|^2 \leq C \eta^2$. Thus, for all $\eta$ small enough in terms of the smoothness of the boundary, we have
\begin{align*}
I_2 = & \int_{B_{\eta/\sqrt{\rho}}' \setminus B_1'} \int_{-1}^{\eta/\rho - \psi_\rho(y')} \frac{\log(1 + t)}{(|t + \psi_\eta(y')|^2 + |y'|^2)^{(N + 2s)/2}}dt dy' \\
= & \int_{B_{\eta/\sqrt{\rho}}' \setminus B_1'} \int_{-1}^{\eta/\rho - \psi_\rho(y')} \frac{\log(1 + t)}{(|t|^2 + |y'|^2)^{(N + 2s)/2}}dt dy' \\
& + \rho \int_{B_{\eta/\sqrt{\rho}}' \setminus B_1'} |y'|^2 \int_{-1}^{\eta/\rho - \psi_\rho(y')} \frac{\log(1 + t)}{(|t|^2 + |y'|^2)^{(N + 2s)/2}}dt dy' \\
=: & I_{21} + I_{22}.
\end{align*}

For $I_{21}$, we notice that for all $|y'| \geq 1$ we have
\begin{align*}
0 \leq & \int_{\eta/\rho - \psi_\rho(y')}^{+\infty} \frac{\log(1 + t)}{(|t|^2 + |y'|^2)^{(N + 2s)/2}}dt \\
\leq & C \int_{\eta/(2\rho)}^{+\infty} \frac{\log(1 + t)}{(t + |y'|)^{N + 2s}}dt \\
= & \frac{C}{-N - 2s + 1} \Big{\{}  \log(1 + t) (t + |y'|)^{-N - 2s + 1} \Big{|}_{\eta/(2\rho)}^{+\infty} - \int_{\eta/(2\rho)}^{+\infty} \frac{dt}{(1 + t)(t + |y'|)^{N + 2s - 1}} \Big{\}} \\
\leq & C \frac{\rho |\log \rho |}{(1 + |y'|)^{N + 2s - 2}},
\end{align*}
from which
\begin{align*}
I_{21} = & \int_{B_{\eta/\sqrt{\rho}}' \setminus B_1'} \int_{-1}^{+\infty} \frac{\log(1 + t)}{(|t|^2 + |y'|^2)^{(N + 2s)/2}}dt dy' + O(\rho |\log \rho |) \int_{B_{\eta/\sqrt{\rho}}' \setminus B_1'} \frac{dy'}{(1 + |y'|)^{N + 2s - 2}} \\
= & \int_{B_{\eta/\sqrt{\rho}}' \setminus B_1'} \int_{-1}^{+\infty} \frac{\log(1 + t)}{(|t|^2 + |y'|^2)^{(N + 2s)/2}}dt dy' + O(\rho^{s + 1/2} |\log \rho |),
\end{align*}
where the $O$-term depends on the smoothness of $\Omega, N, s$ and $\eta$.

Similarly, for $I_{22}$ we have
\begin{align*}
I_{22} = & O(\rho)\int_{B_{\eta/\sqrt{\rho}}' \setminus B_1'} \frac{|y'|^2dy'}{(1 + |y'|^2)^{(N + 2s)/2}} \\
& + \rho \int_{B_{\eta/\sqrt{\rho}}' \setminus B_1'} |y'|^2\int_{1}^{\eta/\rho - \psi_\rho(y')} \frac{\log(1 + t)}{(|t|^2 + |y'|^2)^{(N + 2s)/2}}dt dy' \\
= & O(\rho^{s + 1/2}) + O(\rho) \int_{B_{\eta/\sqrt{\rho}}' \setminus B_1'} |y'|^2 \int_{1}^{2\eta/\rho} \frac{\log(1 + t)}{(t + |y'|)^{N + 2s}}dt dy' ,
\end{align*}
from which, using integration by parts in the last integral, we conclude that
\begin{equation*}
I_{22} = O(\rho^{s + 1/2}) + O(\rho^s |\log \rho|).
\end{equation*}

Gathering the estimates for $I_{21}, I_{22}$, we get
\begin{align}
I_2 = \int_{B_{\eta/\sqrt{\rho}}' \setminus B_1'} \int_{-1}^{+\infty} \frac{\log(1 + t)}{|(t, y')|^{N + 2s}}dt dy' + O(\rho^{s} |\log \rho |).
\end{align}

Finally, for $I_3$, we see that
\begin{align*}
I_3 = \int_{B'_{\eta/\rho} \setminus B'_{\eta/\sqrt{\rho}}} \int_{-1}^{\eta/\rho - \psi_\rho(y')} \frac{\log(1 + t)}{(|t + \psi_\rho(y')|^2 + |y'|^2)^{(N + 2s)/2}}dt dy'.
\end{align*}

Now, since $|y'| \leq \eta/\rho$, we have $|\psi_\rho(y')| \leq C_\Omega \rho |y'|^2 \leq C_\Omega \eta^2/\rho$, from which, by taking $\eta$ small enough we have $\eta/(2\rho) \leq \eta/\rho - \psi_\rho(y') \leq 2\eta/\rho$.

In addition, for $(t, y')$ in the domain of integration, we have
$$
2t\psi_\rho(y') + |\psi_\rho(y')|^2 \leq 4 C_\Omega \eta |y'|^2 + C_\Omega^2 \eta^2 |y'|^2 \leq C \eta |y'|^2,
$$
where $C > 0$ just depends on the smoothness of the boundary. By taking $\eta$ small enough in terms of $C$, by the Mean Value Theorem we have
\begin{align*}
\Big{|}\frac{1}{(|t + \psi_\rho(y')|^2 + |y'|^2)^{(N + 2s)/2}}	 - \frac{1}{|(t, y')|^{N + 2s}} \Big{|} \leq \frac{C|y'|^2}{|(t, y')|^{N + 2s + 2}},
\end{align*}
for some constant $C > 0$ not depending on $\rho, \eta$. Then, we have
\begin{align*}
	I_3 = & \int_{B'_{\eta/\rho} \setminus B'_{\eta/\sqrt{\rho}}} \int_{-1}^{\eta/\rho - \psi_\rho(y')} \frac{\log(1 + t)}{|(t, y')|^{N + 2s}}dt dy' + O(1) \int_{B'_{\eta/\rho} \setminus B'_{\eta/\sqrt{\rho}}} |y'|^2 \int_{-1}^{3\eta/\rho} \frac{|\log(1 + t)|}{|(t, y')|^{N + 2s + 2}}dt dy' \\
	= & \int_{B'_{\eta/\rho} \setminus B'_{\eta/\sqrt{\rho}}} \int_{-1}^{\eta/\rho - \psi_\rho(y')} \frac{\log(1 + t)}{|(t, y')|^{N + 2s}}dt dy' + O(\rho^s |\log \rho|).
	\end{align*}

Thus, joining the estimates for $I_1, I_2, I_3$ above, we conclude that
\begin{align*}
I = & \int_{B'_{\eta/\rho}} \int_{-1}^{\eta/\rho - \psi_\rho(y')} \frac{\log(1 + t)}{|(t, y')|^{N + 2s}}dt dy' + O(\rho^s |\log \rho|) \\
= &  \int_{B'_{\eta/\rho}} \int_{-1}^{+\infty} \frac{\log(1 + t)}{|(t, y')|^{N + 2s}}dt dy' + O(\rho^s |\log \rho|),\\
= & \int_{\R^{N}} \int_{-1}^{+\infty} \frac{\log(1 + t)}{|(t, y')|^{N + 2s}}dt dy' + O(\rho^s |\log \rho|),\\
= & c_s
\end{align*}
where in the second equality we have used that $\eta/\rho - \psi_\rho(y') \geq \eta/(2\rho)$ and
for the third equality we need to establish the following claim
 $$K=\int_{{(B'_{1/\rho})}^c} \int_{-1}^{+\infty} \frac{\log(1 + t)}{|(t, y')|^{N + 2s}}dt dy' =O(\rho^s |\log \rho|).$$
 Recall that $c_s$ is the constat of the one dimensional case son negative.
 Therefore, to finish we just need to establish the last claim.
 For that we will split the integral  
 $$K=\int_{{(B'_{1/\rho})}^c} \int_{-1}^{+\infty} \frac{\log(1 + t)}{|(t, y')|^{N + 2s}}dt dy' =K_1+K_2+K_3,$$
 where, for $i=1,2,3$ we write 
 $$
 K_i=\int_{{(B'_{1/\rho})}^c}J_idy'
 $$ 
with
 $$
 J_1=\int_{-1}^{1} \frac{\log(1 + t)}{|(t, y')|^{N + 2s}}dt, \quad 
J_2=\int_{1}^{1/\sqrt{\rho}} \frac{\log(1 + t)}{|(t, y')|^{N + 2s}}dt, \quad 
 J_3=\int_{1/\sqrt{\rho}}^{+\infty} \frac{\log(1 + t)}{|(t, y')|^{N + 2s}}dt.
 $$
 
To estimate $K_1$ we see that $$J_1\leq \int_{-1}^{1} \frac{|\log(1 + t)|}{|y')|^{N + 2s}}dt,$$
therefore
$$K_1\leq \int_{{(B'_{1/\rho})}^c}\frac{1}{|y'|^{N + 2s}}\leq \rho^{2s+1}$$.
Now for $K_2$ we use Fubbini and a change of variable to get
 $$K_2 =\int_{-1}^{1/\sqrt{\rho}} \frac{\log(1 + t)}{t^{1 + 2s}}\int_{{(B'_{1/(t\rho)})}^c}\frac{1}{(1+|y'|^2)^{\frac{N + 2s}{2}}}dy'dt .$$
 Observe that $t\leq 1/\sqrt{\rho}$, thus $1/\sqrt{\rho}\leq 1/(t\rho)$
 $$K_2 \leq\int_{-1}^{1/\sqrt{\rho}} \frac{\log(1 + t)}{t^{1 + 2s}}\int_{{(B'_{1/(\sqrt{\rho})})}^c}\frac{1}{(1+|y'|^2)^{\frac{N + 2s}{2}}}dy'dt,$$
and then the first integral is finite, thus we get
$$K_2 \leq \rho^{s+1/2}$$

Finally, for $J_3$ we argue similarly as to the estimate $I_{21}$ to get 
$$
J_3\leq \frac{\sqrt{\rho}|\log (\rho)|}{(1+|y'|)^{n+2s-2}},
$$
from where $K_3\leq \rho^{2s-1/2}|\log (\rho)|$. Thus, we get the claim and the proof follows.

\smallskip

\end{proof}

\begin{lema}\label{lemabarrera1}
Assume $s \in (1/2,1)$. For $R, \mu > 0$ and $\beta \in (0,2s- 1)$, consider the function
$$
w_R(x) = \mu (R - R^{\frac{\beta}{\gamma} + 1} d^\beta(x))_+, \quad x \in \Omega.
$$

Then, for all $\beta$, there exists $\mu_0 > 0$ small (depending on $\beta, N, p$ and $s$), such that for all $\mu \leq \mu_0$ and all $R$, the function $w_R$ satisfies
\begin{equation*}
\Ls w_R(x) + |D w_R(x)|^m \leq \frac{c(\beta)}{4} \mu  R^{\frac{\beta}{\gamma} + 1} d^{\beta - 2s}(x) \quad \mbox{in} \ \Omega \setminus \Omega_{R^{-\frac{1}{\gamma}}},
\end{equation*}
where $c = c(\beta) < 0$ is the constant of Lemma~\ref{lemabarrera0}. 
\end{lema}

\begin{proof}
Notice that $w_R(x) = 0$ for all $x$ such that $d(x) \geq R^{-\frac{1}{\gamma}}$. Thus, without loss of generality, we assume $R$ is large enough such that $R^{-\frac{1}{\gamma}} < \delta_0$ in Lemma~\ref{lemabarrera0}. In fact, using that lemma we have
\begin{equation*}
\Ls w_R(x) \leq \mu R^{\frac{\beta}{\gamma} + 1} c(\beta) d^{\beta - 2s}(x)/2.
\end{equation*}

On the other hand, a direct computation leads us to
\begin{equation*}
D w_R(x) = \beta \mu R^{\frac{\beta}{\gamma} + 1} d^{\beta - 1}(x) Dd(x),
\end{equation*}
and from here we get
\begin{align*}
& \Ls w_r(x) + |Dw_R(x)|^m \\
\leq & \ \mu R^{\frac{\beta}{\gamma} + 1} d^{\beta - 2s}(x) \Big{(} c(\beta) + \beta^m \mu^{m-1} R^{\frac{(\beta + \gamma)(m-1)}{\gamma}} d^{m(\beta - 1) - \beta + 2s}(x)\Big{)}.
\end{align*}

Notice that since $m < 2s$ and $2s > 1$, we have $m(\beta - 1) - \beta + 2s > 0$.
Considering that $d(x) \leq R^{-1/\gamma}$, and using that $m(-\gamma - 1) = -\gamma - 2s$, we get
\begin{align*}
\Ls w_r(x) + |Dw_R(x)|^m \leq \mu R^{\frac{\beta}{\gamma} + 1} d^{\beta - 2s}(x) \Big{(} c(\beta)/2 + \beta^m \mu^{m-1} \Big{)}.
\end{align*}

Then, taking $\mu$ small enough, we conclude the result.
\end{proof}

\subsection{Regularity estimates.}

We recall here that for an open set $\mathcal O \subset \R^N$, we denote $\mathcal O_\delta = \{ x \in \mathcal O : \mathrm{dist}(x, \partial \mathcal O) > \delta \}$.
The following result is a straightforward adaptation of similar results provided, for instance, in~\cite{BCCI, BT}, and~\cite{BCCI14}. We provide the details for completeness.
\begin{lema}\label{lemaLipbounded}
Let $s \in (1/2, 1)$, $\Omega \subset \Omega \subseteq \R^N$ with $\Omega'$ open and bounded, with $C^2$ boundary. 
	Let $\lambda \in [0,1]$, $T > 0$, $f \in L^\infty_{loc}(\Omega)$ and $u \in C(\Omega') \cap L^\infty(\Omega)$ be a viscosity solution to the equation
	\begin{equation}\label{equA}
		\lambda u  + \Ls u + T |Du|^m = f \quad \mbox{in} \ \Omega'.
	\end{equation}

\smallskip	

\noindent	
$(1)$ Interior Lipschitz estimates: For each $\delta \in (0,1)$ small enough, there exists $C > 0$ depending on $N, T, m, s, \delta, \| u \|_{L^\infty(\Omega)}$ and $\| f \|_{L^\infty(\Omega'_{\delta})}$, but not on $\mathrm{diam}(\Omega)$ such that
	\begin{align*}
	[u]_{C^{0,1}(\Omega'_{4\delta})} \leq C. 
	\end{align*}
	
\smallskip

\noindent
$(2)$ Interior equicontinuity independent of $T$:	 For each $\delta \in (0,1)$ small enough, there exists $C > 0$ depending on $N, m, s, \delta, \| u \|_{L^\infty(\Omega)}$ and $\| f \|_{L^\infty(\Omega'_{\delta})}$, but not on $\mathrm{diam}(\Omega)$ nor $T$ such that
	\begin{align*}
[u]_{C^{0,\frac{1}{2}}(\Omega'_{4\delta})} \leq C.
	\end{align*}
\end{lema}

\begin{proof}
	With the convention $u = u^+ - u^-$, replacing $f$ by $f + \lambda \sup_\Omega \{ u^- \} \in L^\infty(\Omega)$, we can assume $u \geq 0$ in $\Omega$.
	
For $t_0, c_0 \in (0,1)$ small, consider the function $\varphi(t) = t - c_0t^{1 + \alpha}$ for $t \in [0,t_0]$ and $\varphi(t) = \varphi(t_0)$ for $t > t_0$, and such that $\varphi(t_0) > 0$.	
	
From now on, we assume $\delta_0 < t_0$. Let $\delta < \delta_0$, and let $\eta$ be a smooth cut-off function with $0 \leq \eta \leq 1$, $\eta = 0$ in $\Omega_{4\delta}$, $\eta = 1$ in $\Omega \setminus \Omega_{3\delta}$.

Let $C_1 = \| u \|_{L^\infty(\Omega)} + 1$. We claim that there exists $L > 1$ (depending on $\delta, \| u \|_\infty$ and the data) such that
\begin{equation*}
\sup_{x,y \in \Omega} \{ u(x) - u(y) - L\varphi(|x - y|) - C_1 \eta(x) \} \leq 0.
\end{equation*}

This concludes the Lipschitz estimate for $u$ on $\Omega_{4\delta}$.

By contradiction, we assume the existence of $x,y \in \Omega$ such that
$$
\Phi(x,y) := u(x) - u(y) - L\varphi(|x - y|) - C_1\eta(x) > 0,
$$
and from the definition of $\eta$ and $\delta_0$ we have that all such $x$ must satisfy $x \in \Omega_{3\delta}$. Moreover, if $y \in \Omega \setminus \Omega_{2\delta}$, using that $\varphi$ is nondecreasing we have
\begin{equation*}
C_1 - L \varphi(\delta) \geq u(x) - u(y) - L\varphi(|x - y|) - C_1\eta(x) > 0,
\end{equation*}
and this is impossible if we take  
$$
L > \frac{C_1}{\varphi(\delta)}.
$$

We constraint $L$ to this regime, from which we obtain that the supremum above is positive, and it is attained at some $(\bar x, \bar y) \in \Omega_{2\delta} \times \Omega_{2\delta}$. Notice in addition that necessarily $\bar x \neq \bar y$ and moreover
%
$$
L \varphi(|\bar x - \bar y|) \leq 2 \| u \|_{L^\infty(\Omega)}.
$$
from which, by taking $L \geq 2C_1 \delta^{-1}$ we can conclude that
\begin{equation}\label{x-y}
|\bar x - \bar y| \leq C_0 L^{-1} \| u \|_{L^\infty(\Omega)},
\end{equation}
for some $C_0$ depending on $c_0, t_0$, but not on $L$ nor $\delta$.

Then, writing $\phi(x,y) := L\varphi(|x - y|) + C_1\eta(x)$, we can use $x \mapsto u(\bar y) +\phi(x, \bar y)$ as a test function for $u$ at $\bar x$, and $y \mapsto u(\bar x) - \phi(\bar x, y)$ as a test function for $u$ at $\bar y$. The arguments in~\cite[Lemma 3.1]{CS09} apply for this equation and we can evaluate the nonlocal operators pointwise in the viscosity inequalities  to conclude that
\begin{equation}\label{testing}
\begin{split}
\lambda u(\bar x) + \Ls u(\bar x) + T |D_x \phi(\bar x, \bar y)|^m & \leq f(\bar x), \\
\lambda u(\bar y) + \Ls u(\bar y) + T |D_x \phi(\bar x, \bar y)|^m & \geq f(\bar y).
\end{split}
\end{equation}

Then, we subtract both inequalities and using that $u(\bar x) \geq u(\bar y)$, we arrive at
\begin{equation}\label{testLip}
\begin{split}
& H \leq I + H + 2 \| f \|_{L^\infty(\Omega_\delta)}, \quad \mbox{with}, \\
& I := -\Ls u(\bar x) + \Ls u(\bar y), \\
& H := -T(|D_x \phi(\bar x, \bar y)|^m - |D_x \phi(\bar x, \bar y)|^m).
\end{split}
\end{equation}

We estimate each term separately. In fact, a change of variable leads us to
\begin{align*}
I = & C_{N,s} \mathrm{P.V.} \int_{\Omega - \bar x} \frac{u(\bar x + z) - u(\bar x)}{|z|^{N + 2s}}dz - C_{N,s} \mathrm{P.V.} \int_{\Omega - \bar y} \frac{u(\bar y + z) - u(\bar y)}{|z|^{N + 2s}}dz.
\end{align*}

Denote $\bar a = \bar x - \bar y$. Taking $L$ large enough in~\eqref{x-y}  in order that $|\bar a | \leq \delta/8$, we conclude that $\bar x, \bar y$ are at distance at least $\delta/8$ from the sets $\Omega \setminus (\Omega \pm \bar a)$. Using that $u \geq 0$, we can estimate $I$ as
\begin{align*}
I \leq & C_{N,s}\mathrm{P.V.} \int_{(\Omega - \bar x)\cap (\Omega - \bar y)} \frac{u(\bar x + z) - u(\bar x)}{|z|^{N + 2s}}dz - C_{N,s}\mathrm{P.V.} \int_{(\Omega - \bar y) \cap (\Omega - \bar x)} \frac{u(\bar y + z) - u(\bar y)}{|z|^{N + 2s}}dz \\
& + C_{N,s} \int_{\Omega \setminus (\Omega + \bar a)} u(z) |\bar x - z|^{-(N + 2s)}dz + C_{N,s} \| u \|_{L^\infty(\Omega_\delta)} \int_{\Omega \setminus (\Omega - \bar a)} |z - \bar y|^{-(N + 2s)}dz.
\end{align*}

This leads us to the existence of $C > 0$ not depending on $\delta$ nor $L$ such that
\begin{align*}
& \int_{\Omega \setminus (\Omega + \bar a)} u(z) |\bar x - z|^{-(N + 2s)}dz \leq C \| u \|_{L^\infty(\Omega)} \delta^{-2s}, \\
& \int_{\Omega \setminus (\Omega - \bar a)} |z - \bar y|^{-(N + 2s)}dz \leq C \delta^{-2s},
\end{align*}
from which we conclude the existence of $C$ just depending on $N,s$ such that
\begin{align*}
I \leq I_1 + I_2 + C (\| u \|_{L^\infty(\Omega)} + 1)\delta^{-2s},
\end{align*}
with 
\begin{align*}
I_1 := & C_{N,s} \mathrm{P.V.} \int_{(\Omega - \bar x)\cap (\Omega - \bar y)} \frac{u(\bar x + z) - u(\bar x)}{|z|^{N + 2s}}dz, \\
I_2 := & C_{N,s}\mathrm{P.V.} \int_{(\Omega - \bar y) \cap (\Omega - \bar x)} \frac{u(\bar y + z) - u(\bar y)}{|z|^{N + 2s}}dz
\end{align*}

At this point, we introduce the set
$$
\mathcal C = \mathcal C_{r, \rho} := \{ z \in B_{r} : |\bar a \cdot z| \geq (1 - \rho) |\bar a| |z| \},
$$ 
for some $r, \rho > 0$ to be fixed small. In fact, following Corollary 9 in~\cite{BCCI}, we pick
$$
r = r_0 |\bar a|^{1 + \alpha}, \ \rho = \rho_0 |\bar a|^{2\alpha},
$$
for some $r_0, \rho_0 \in (0,1)$ small but independent of $\delta, L$. By enlarging $L$ if necessary, we have that $\mathcal C \subset (\Omega - \bar x) \cap (\Omega - \bar y)$, from which, using Lemmas 2.2 and 2.4 in~\cite{BQT} we obtain the existence of $c > 0$ just depending on $\alpha, N, s$ such that
\begin{align*}
I_1 - I_2 \leq -cL |\bar a|^{1 - 2s + \alpha(N + 2 - 2s)} + C_1 \mathrm{P.V.} \int_{(\Omega - \bar x) \cap (\Omega - \bar y)} [\eta(\bar x + z) - \eta(\bar x)]  |z|^{-(N + 2s)}dz
\end{align*}

Using that $\| D \eta \| \leq C \delta^{-1}, \|D^2 \eta \|_\infty \leq C \delta^{-2}$ for some constant $C > 0$ just depending on the dimension and $\Omega$, we have the integral term in the r.h.s. of the last inequality can be bounded above by $C C_1 \delta^{-2}$. Hence,
\begin{equation*}
I_1 - I_2 \leq -cL |\bar a|^{1 - 2s + \alpha(N + 2 - 2s)} + C C_1 \delta^{-2},
\end{equation*}
from which we arrive at
\begin{align}\label{ILip}
I \leq -cL |\bar a|^{1 - 2s + \alpha(N + 2 - 2s)}  + C C_1 \delta^{-2}.
\end{align}

On the other hand, denoting $\hat a = \bar a / |\bar a|$, we see that
\begin{align*}
H = & -T(|L \varphi'(|\bar a|) \hat{a} + D \eta (\bar x)|^m - |-L \varphi'(|\bar a|) \hat{a}|^m)\\
= & -T L^{m} |\varphi' (|\bar a|)|^m \Big{(} |\hat a - L^{-1} (\varphi'(|\bar a|))^{-1} D\eta(\bar x)|^m - 1\Big{)},
\end{align*}
and by taking $L$ large enough in terms of $\delta$, we conclude that
$$
H \leq C C_1 T L^{m - 1} \delta^{- 1},
$$
for some $C > 0$ not depending on $\delta$ nor $L$. Gathering this inequality,~\eqref{ILip} and replacing them into~\eqref{testLip} we arrive at
\begin{equation*}
L |\bar a|^{1 - 2s + \alpha(N + 2 - 2s)}  \leq C C_1 \delta^{-2} + C C_1 T L^{m - 1} \delta^{ - 1} + 2\| f \|_{L^\infty(\Omega_\delta)}.
\end{equation*}

Take $\epsilon \in (0,1)$. Then, we can take $\alpha$ small order to $1 - 2s + \alpha(N + 2 - 2s) < 1 - 2s + \epsilon < 0$. By~\eqref{x-y}, we get that
\begin{align*}
	cL^{2s - \epsilon} C_1^{1 - 2s + \epsilon}  \leq C C_1 \delta^{-2} + C C_1 T L^{m - 1} \delta^{ - 1} + 2\| f \|_{L^\infty(\Omega_\delta)},
\end{align*}
for some $c > 0$ not depending on $L, \delta$ nor $u$. Thus, since $m \leq 2s$, taking $L$ large in terms of $\delta, T$ and $\| f \|_{L^\infty(\Omega_\delta)}$ we arrive at a contradiction.

\medskip

\noindent
$(2)$ We follow closely the arguments in~\cite[Lemma 1]{BCCI14}. We consider $\mu \in (0,1)$ and denote $\tilde u = \mu u$, that solves the problem
\begin{equation*}
\lambda \tilde u + \Ls \tilde u + \mu^{1-m} |D\tilde u|^m = \mu f \quad \mbox{in} \ \Omega'.
\end{equation*}

Now, for $L_\mu = L/(1 - \mu)$ with $L > 1$, we consider the quantity
\begin{align*}
M_\mu = \sup_{x, y \in \Omega} \{ \tilde u(x) - u(y) - L_\mu \varphi(|x - y|) - C_1 \eta(x) \}.
\end{align*}

We will prove that there exists $L_0 >1$ not depending on $\mu$ such that
\begin{equation*}
M_\mu \leq (\| u \|_{L^\infty(\Omega)} + 1) (1 - \mu),
\end{equation*}
for all $L \geq L_0$ and all $\mu \in (0,1)$. Notice that if the previous inequality holds, after modifying $L$ by a constant not depending on $\mu$, we arrive at
\begin{equation*}
|u(x) - u(y)| \leq \frac{L}{1 - \mu} |x - y| + (2\| u \|_{L^\infty(\Omega)} + 1) (1 - \mu), \quad x, y \in \Omega'_{4\delta},
\end{equation*}
and taking infimum in $\mu \in (0,1)$, we conclude that $u$ is $C^{0, \frac{1}{2}}(\Omega_{4\delta})$.

We proceed as before, assuming that $M_\mu > (\| u \|_{L^\infty(\Omega)} + 1)(1 - \mu)$. Then, denoting 
$$
\tilde \Phi(x,y) = \tilde u(x) - u(y) - L\varphi(|x - y|) - C_1 \eta(x),
$$ 
we have the existence of $\bar x, \bar y \in \Omega_{2\delta}$ such that that $M_\mu = \sup_{x, y \in \Omega} \tilde \Phi(\bar x, \bar y) = \tilde \Phi(\bar x, \bar y)$, and therefore we get
$$
(\| u \|_{L^\infty(\Omega)} + 1)(1 - \mu) < \mu |u(\bar x) - u(\bar y)| + \| u \|_{L^\infty(\Omega)}(1 - \mu),
$$
from which $\bar x \neq \bar y$ for all $\mu \in (0,1)$.

Then, from this point we follow exactly the same proof as before, where this time the viscosity inequalities in~\eqref{testing} take the form
\begin{equation*}
\begin{split}
\lambda \tilde u(\bar x) + \Ls \tilde u(\bar x) + T \mu^{1 - m} |D_x \phi(\bar x, \bar y)|^m & \leq \mu f(\bar x), \\
\lambda u(\bar y) + \Ls u(\bar y) +|D_x \phi(\bar x, \bar y)|^m & \geq f(\bar y),
\end{split}
\end{equation*}
in view of the equation solver by $\tilde u$, and from here, subtracting both inequalities, since $\tilde u(\bar x) - u(\bar y) \geq (\| u \|_{L^\infty(\Omega)} + 1)(1-\mu)$, inequality~\eqref{testLip} takes the form
\begin{equation}\label{testLip2}
\begin{split}
& 0 \leq I + H + (1 - \mu) \Big{(} \| f \|_{L^\infty(\Omega_\delta)} + \lambda (\| u \|_{L^\infty(\Omega)} + 1) \Big{)}, \quad \mbox{with}, \\
& I := -\Ls \tilde u(\bar x) + \Ls u(\bar y), \\
& H := -T(\mu^{1 - m}|D_x \phi(\bar x, \bar y)|^m - |D_x \phi(\bar x, \bar y)|^m).
\end{split}
\end{equation}

The estimate for $I$ is the same as above, and the constants involved do not depend on $\mu$. On the other hand, recalling that $\bar a = \bar x - \bar y$, for $H$ we see that
\begin{align*}
H 
= & -T L^{m} |\varphi' (|\bar a|)|^m \Big{(} \mu^{1 - m} |\hat a - L^{-1} (\varphi'(|\bar a|))^{-1} D\eta(\bar x)|^m - 1\Big{)}.
\end{align*}

By~\eqref{x-y} for all $L$ large enough we have $\varphi'(|\bar a|) > 1/2$ and by the estimates for $D\eta$ we have the estimates
\begin{align*}
|\hat a - L^{-1} (\varphi'(|\bar a|))^{-1} D\eta(\bar x)|^m \geq (1 - CL^{-1}\delta^{-1})^m \geq 1 - CmL^{-1}\delta^{-1},
\end{align*}
for some $C > 1$ not depending on $L, \delta$ nor $m$. Using this, we arrive at
\begin{align*}
H \leq & T L^{m} |\varphi' (|\bar a|)|^m (1 - \mu^{1 - m} + CmL^{-1}\delta^{-1})\\
\leq & T L^{m} |\varphi' (|\bar a|)|^m \Big{(} c_m(\mu - 1) + CmL^{-1}\delta^{-1}\Big{)},
\end{align*}
for some $c_m > 0$ if $m > 1$. From here, taking $L \geq \tilde c_m \delta^{-1} (1 - \mu)^{-1}$ for some $\tilde c_m$ large enough, we conclude that
$H \leq 0$. Gathering the estimates as in the previous case, we arrive at a contradiction by taking $L$ large in terms of the data, $\delta$ and $1 - \mu$, but not on $T$. This concludes the proof.
%
%
\end{proof}

Recall that we denote $d(x) = \mathrm{dist}(x, \partial \Omega)$ for $x$ close to the boundary and it is extended as a smooth, strictly positive function in the whole of $\Omega$.

The above lemma leads us to the following estimate for $L^1$ solutions.
\begin{prop}\label{lemaequiLip}
	Let $\lambda \in [0,1]$, $f \in L^\infty_{loc}(\Omega)$ and $u \in C(\Omega)$ be a viscosity solution to the equation~\eqref{eq},
	bounded from below, satisfying that $u(x) \leq C_0 d^{-\gamma}(x)$ for some $C_0 > 0$. Then, $u \in C^{1, \alpha}(\Omega)$ for some $\alpha \in (0,1)$ and for each $\delta > 0$ small enough we have the estimate
	\begin{align*}
		[u]_{C^{0,1}(\Omega_\delta)} \leq C \delta^{-\gamma - 1}.
	\end{align*}
	
	Moreover, if $f$ is H\"older continuous in $\Omega$, then $u \in C^{2s + \alpha}(\Omega)$ for some $\alpha \in (0,1)$.
\end{prop}

\begin{proof}
%
%
%
%
The $C^{1, \alpha}$ regularity is a consequence of the interior Lipschitz estimates of Lemma~\ref{lemaLipbounded} together with interior $C^{1, \alpha}$ estimates for the fractional Laplacian, see~\cite[Theorem 12.1]{CS09}. If $f$ is H\"older continuous, higher-order regularity comes from~\cite[Theorem 1.1]{S15}. In the application of these results, the fact that the current operator is the censored fractional Laplacian plays no significant role, since estimates are interior, and we can extend $u$ as a function of $\R^N$ in a $C^\alpha$ fashion, eventually redefining it in a small neighborhood of $\partial \Omega$.

Now we deal with the gradient bounds depending on the distance to the boundary.
Let $x_0 \in \Omega$, denote $\rho = d(x_0)/4 > 0$ and let $\tilde u = u \chi_{\Omega_\rho} \in C(\Omega_\rho) \cap L^\infty(\Omega)$. We focus on the case $x_0$ is close to the boundary.

It is easy to see that for all $x \in \Omega$ with $d(x) > \rho$ we have
\begin{equation*}
-\Ls \tilde u = -\Ls u - C_{N, s}\int_{\Omega \setminus \Omega_\rho} \frac{u(z)}{|z - x|^{N + 2s}}dz,
\end{equation*}
from which $\tilde u$ solves the problem
\begin{equation}\label{eqtildeu}
\lambda u + \Ls \tilde u + |D\tilde u|^m = f + h_\rho \quad \mbox{in} \ \Omega_{2 \rho},
\end{equation}
where
$$
h_\rho(x) = C_{N,s}\int_{\Omega \setminus \Omega_\rho} \frac{u(z)}{|z - x|^{N + 2s}}dz. 
$$

Notice that the function $h$ is bounded and smooth in $\Omega \setminus \Omega_{2 \rho}$. We claim that there exists $C > 0$ depending on $N, s, \gamma$ and $\Omega$ such that 
$$
\| h \|_{L^\infty(\Omega_{2\rho})} \leq C \rho^{-\gamma - 2s}.
$$

 In fact, given $x \in \Omega_{2\rho}$ close enough to $\partial \Omega_\rho$ (and hence, close to $\partial \Omega$), and after translation and rotation, we can assume that $x = (0', x_N)$, that the projection of $x$ onto $\partial \Omega$ is the origin, and that on a neighborhood of the origin, the boundary is parametrized by a $C^2$ function $\psi: B_{r_0}' \subset \R^{N -1} \to \R$ such that $\psi(0') = 0, D\psi(0') = 0'$ and therefore $|D^2 \psi(y')| \leq C_\Omega |y'|^2$ for all $|y'| \leq r_0$, for some constants $r_0, C_\Omega$ just depending on $\Omega$.

Then, for an open set $\mathcal U \subset \R^N$ such that $B_{r_0/2} \subset \mathcal U$, we have
$$
\Omega_{\rho} \cap \mathcal U \subseteq  \{ (y', \psi(y') + t) \in \R^N : y' \in B_{r_0}', \ t \in (0, 3\rho/2) \}.  
$$

From now on, we assume $\rho < r_0/8$.

For $x$ such that $d(x) \geq r_0$, we have 
$$
|h_\rho(x)| \leq C r_0^{-(N + 2s)} \int_{\Omega \setminus \Omega_\rho} d^\gamma(z)dz \leq C,
$$
for some $C > 0$ depending on $r_0, N, s$ and $\Omega$.

For $d(x) < r_0$, we split the integral as
$$
h_\rho(x) = C_{N,s}\int_{(\Omega \setminus \Omega_\rho) \cap \mathcal U} \frac{u(z)}{|z - x|^{N + 2s}}dz + C_{N,s}\int_{(\Omega \setminus \Omega_\rho) \cap \mathcal U^c} \frac{u(z)}{|z - x|^{N + 2s}}dz,
$$
and using similar arguments as the case $d(x) \geq r_0$ we arrive at
\begin{align*}
h_\rho(x)	= C_{N,s}\int_{(\Omega \setminus \Omega_\rho) \cap \mathcal U} \frac{u(z)}{|z - x|^{N + 2s}}dz + O(1) =: J + O(1).
\end{align*}

Using the estimates for $u$ and the change of variables $z = (y', \psi(y') + t), \ |y'| \leq r_0, \ t \in (0, 3\rho/2)$, we get
$$
J \leq C_{N, s} C_\Omega \int_{B_{r_0}'} \int_{0}^{3\rho/2} \frac{t^{-\gamma}}{(|y'|^2 + (x_N - \psi(y') - t)^2)^{\frac{N + 2s}{2}}}dt dy'.
$$

Now, by the properties of the map $\psi$, we can assume that $|\psi(y')| \leq \rho/4$ if we shorten $\rho$ is necessary. Thus, recalling that $_N \geq 2\rho$, we have $|x_N - \psi(y') - t| \geq \rho/8$ for $|y'| \leq \rho$. Thus
\begin{align*}
J \leq & C \int_{|y'|\leq \rho} \int_{0}^{3\rho/2} \frac{t^{-\gamma}}{(|y'|^2 + \rho^2)^{\frac{N + 2s}{2}}}dt dy' + C \int_{\rho \leq |y'|\leq r_0} \int_{0}^{3\rho/2} \frac{t^{-\gamma}}{|y'|^{N + 2s}}dt dy' \\
\leq & C \rho^{1-\gamma} \Big{(} \rho^{-1 - 2s} \int_{B_1'} \frac{d z'}{(|z'|^2 + 1)^{\frac{N + 2s}{2}}} + \rho^{-1 - 2s} \int_{|z'| \geq 1} \frac{dz'}{|z'|^{N + 2s}} \Big{)} \\
\leq & C \rho^{-\gamma - 2s},
\end{align*}
and the claim follows.

\medskip

Now, we prove the gradient estimates. Let $\Omega^\rho := \rho^{-1} (\Omega - x_0)$ and define
$$
v(y) = \rho^{\gamma} \tilde u(x_0 + \rho y), \quad y \in \Omega^\rho,
$$
where $\tilde u = u \chi_{\Omega_\rho}$ is as the begining of the proof. 
Notice that $v$ is bounded, uniformly on $\rho$. 

Moreover, we see that
$$
B_2 \subset \rho^{-1} (\Omega_{2 \rho} - x_0) \subset \Omega^\rho,
$$
with $\mathrm{dist}(B_2, \partial \Omega^\rho) \geq 2$.
By~\eqref{eqtildeu} together with the fact that $\gamma + 2s = (\gamma + 1)m$, we can write
\begin{align*}
(-\Delta)_{\Omega^\rho}^s v(y) + |D v(y)|^m =& \rho^{\gamma + 2s} \Ls \tilde u(x_0 + \rho y) + \rho^{(\gamma + 1)m} |D \tilde u(x_0 + \rho y)|^m
\\ =& \rho^{\gamma + 2s} \Big{(} \Ls \tilde u(x_0 + \rho y) +  |D \tilde u(x_0 + \rho y)|^m \Big{)} \\
= & \rho^{\gamma + 2s} (-\lambda \rho^{-\gamma}v(y) + f(x_0 + \rho y) + h_\rho(x_0 + \rho y)),
\end{align*}
from which we conclude that
$$
\lambda \rho^{2s} v + (-\Delta)_{\Omega^\rho}^s v + |D v|^m = \tilde f \quad \mbox{in} \ B_2,
$$
where $\tilde f(y) := \rho^{\gamma + 2s} (f(x_0 + \rho y) + h_\rho(x_0 + \rho y))$. It is easy to see that $\tilde f \in C(\bar B_2)$, with $L^\infty$ estimates independent of $\rho$. Hence, we can use interior estimates given in Lemma~\ref{lemaLipbounded} to conclude that
$$
\| Dv \|_{L^\infty(B_{1})} \leq C,
$$
for some constant not depending on $\rho$. Coming back to $\tilde u$ we conclude that
$$
\| D u \|_{L^\infty(B_{\rho}(x_0))} = \| D \tilde u \|_{L^\infty(B_{\rho}(x_0))} \leq C \rho^{-\gamma - 1},
$$
from which the result follows by the fact that $\rho = 4 d(x_0)$.
\end{proof}

%


\section{The discounted problem.}
\label{secdiscount}

This section is devoted to the proof of Theorem~\ref{teo1}. We start with the following
\begin{prop}\label{propexistence} 
Under the assumptions of Theorem~\ref{teo1}, there exists a viscosity solution $u = u_\lambda$ to~\eqref{eq} satisfying~\eqref{blow-up}. 

Moreover, there exists constants $0 < c_1 \leq c_2$ and $C > 0$ such that 
$$
-C\lambda^{-1} + c_1 g_{\gamma}(x)\leq u(x) \leq c_2 g_{\gamma}(x) + C \lambda^{-1},
$$ 
for all $x \in \Omega$ close to the boundary. 

This solution is minimal in the sense that another blow-up solution $v$ to~\eqref{eq} satisfies $u \leq v$ in $\Omega$.
\end{prop}

\begin{proof}
Let $\delta_0$ be the constant in Lemma~\ref{lemabarrera0}, and $C_0 > 0$ the  unique constant solving \eqref{defc0}.

For some $C^* > 0$ to be fixed and each $\epsilon \in (0,1)$, we consider the function
$$
W(x) = (C_0 + \epsilon) g_{\gamma}(x) + \lambda^{-1} C^* , \quad x \in \Omega.
$$

For $d(x) < \delta_0$, using Lemma~\ref{lemabarrera0} we have the existence of a constant $c_\epsilon > 0$ such that
\begin{align*}
\lambda W + \Ls W + |DW|^m \geq (C_1 + c_\epsilon) d(x)^{-\gamma - 2s} + C^* 
\end{align*}
and then we can take $\delta_\epsilon < \delta_0$ such that $W$ is a supersolution to~\eqref{eq} in $\Omega \setminus \Omega_{\delta_\epsilon}$, in view of~\eqref{compor-f}. Now, for such $\epsilon$, we can find $C^* > 0$  large enough (depending on the $C^{2s + \alpha}(\Omega_{\delta_\epsilon/2})$ estimates for $g_{\gamma}$) such that $W$ is a supersolution to~\eqref{eq} in $\Omega$. 

For each $\lambda>0$ and $R\in\N$, denote $u_{R}$ the unique solution to given of 
\begin{equation} \label{R}
\left \{ \begin{array}{rll}
\lambda u + \Ls u + |Du|^m & = \min \{ f, R \}  \quad  &\text{in }\Omega \\
u & =  R \quad &  \text{in }\partial\Omega,
\end{array} \right .
\end{equation}

The existence of such a problem follows from a standard Perron Method. 
In fact, by similar arguments as above, we have that for $\beta \in (0, 2s - 1)$ fixed, the function $\bar U_R = R + \min \{ A_R d^\beta(x), B_R \}$ is a supersolution to the problem for $0 < A_R < B_R$ large enough depending on $R$, and similarly the function $\underline U_R = R + \max \{ -A_R d^\beta(x), -B_R \}$ is a subsolution for the problem by suitable choice of $A_R, B_R$. Since comparison principle holds for sub and supersolutions which are ordered on the boundary, we conclude the existence of $u_R \in C(\bar \Omega) \cap C^{1, \alpha}_{loc}(\Omega)$.

Also by comparison, we have $u_R \leq u_{R + 1} \leq W$ in $\Omega$ for all $R$. Thus, $\{ u_R \}_R$ is uniformly locally bounded in $\Omega$. By the regularity estimates of Section 2.2 , it is also uniformly $C^\alpha_{loc}(\Omega)$ for some $\alpha > 0$, from which, by Arzela-Ascoli Theorem and stability of viscosity solutions, we conclude the existence of a function $u \in C^\alpha(\Omega)$ solving~\eqref{eq} and such that $u_{R} \nearrow u$ locally uniform in $\Omega$ as $R \nearrow +\infty$.	
The upper bound, in the proposition statement, follows from the upper barrier $W$.

To get the lower bound, we divide in cases: for $m<2s$, we follow closely the arguments of~\cite{RT}, Theorem 2.1  together with the barrier of Lemma~\ref{lemabarrera1}. We provide the details here for completeness. Let $w_{\mu, R}$ be as in Lemma~\ref{lemabarrera1}. Taking $\mu \in (0,1)$ small enough, we have $w_{\mu, R}$ is a viscosity subsolution to the Dirichlet problem~\eqref{R} for all $R$ (here we have used that $f$ is nonnegative). Then, by comparison we have
$$
 w_{\mu, R} \leq u_{R} \quad \mbox{in} \quad\ \Omega.
$$

For $a > 0$, denote 
$$
\theta_R(a) = (R - R^{\frac{\beta}{\gamma} + 1}a^{\beta})_+.
$$

Notice that this function attains its positive maximum at 
$$
R(a) = ((\beta + \gamma) \gamma^{-1})^{-\frac{\gamma}{\beta}} a^{-\gamma}, 
$$
with maximum value
$$
\theta_a(R(a)) = \frac{\beta}{\gamma} \Big{(} 1 + \frac{\beta}{\gamma} \Big{)}^{-1 - \frac{\gamma}{\beta}} a^{-\gamma}. 
$$

Since $u \geq u_{R}$ in $\Omega$ for all $R$, for all $x$ and all $R$ we have
$$
u(x)  \geq \mu (R - R^{\frac{\beta}{\gamma} + 1}d^\beta(x))_+ ,
$$
and we take the maximum value of $\theta_{d(x)}$ in this last expression to conclude the existence of a constant $c = c(\beta, \gamma) > 0$ such that
$$
u(x) \geq \mu c d^{-\gamma}(x) 
$$
from which the lower bound for the boundary blow-up follows in the case $m < 2s$. 

\smallskip

In the case $m=2s$, for $\beta \in (0, 2s-1)$ we consider the function
$$
w_\beta(x) = \frac{\alpha}{\beta}(1-d(x)^\beta),
$$
and notice that $w_\beta(x) = \alpha \beta^{-1}(1 - e^{\beta \log(d(x))}) \leq -\alpha \log(d(x))$ for all $x \in \Omega$.

Notice that by the regularity  of $c(\beta)$ in Lemma~\ref{lemabarrera0}, there exists $\beta_0 \in (0,2s - 1)$ such that  $c(\beta)/\beta\leq c'(0)/2<0$  for all $\beta \in (0, \beta_0)$ . Using this and Lemma~\ref{lemabarrera0}, for each $\delta < \delta_0$ and $x \in \Omega \setminus \Omega_{\delta}$ we see that
 \begin{align*}
\lambda w_\beta + \Ls w_\beta(x) + |D w_\beta(x)|^m \leq & \alpha d^{\beta-2s} \Big{(} -\lambda d^{2s - \beta}(x)  \log(d(x)) + \frac{c(\beta)}{\beta}+\alpha^{m-1}d^{\beta(2s-1)}(x) \Big{)} \\
\leq & \alpha d^{\beta-2s} (\lambda C\delta^{\theta} - c'(0)/2 + \alpha^{m - 1} \delta^{\beta(2s - 1)}),
\end{align*}
for some $\theta \in (0,1)$ and $C > 0$ depending on $\theta$ but not on $\delta$.

Then, taking $\alpha$ is such that $\alpha^{m-1}=-c'(0)/4$ and $\delta$ small enough such that $\lambda C \delta^\theta \leq c'(0)/4$, we arrive at
\begin{equation*}
\lambda w + \Ls w_\beta(x) + |D w_\beta(x)|^m \leq  0 \quad \mbox{in} \ \Omega \setminus \Omega_{\delta}.
\end{equation*}

Fixed $\delta$ in that way, there exists $C_\delta > 0$ such that for all  $R\in \N$ the function
$$
w_R(x)=\frac{\alpha}{1/R}(1-d^{1/R}(x))-C_\delta
$$ 
is a subsolution to~\eqref{eq} (recall $f \geq 0$). Thus,
$w_R\leq u$ and taking $R \to \infty$ we find the lower bound for $u$.

Finally, the minimality of $u$ follows by comparison, since every solution to~\eqref{eq} satisfying~\eqref{blow-up} is a supersolution $v$ for problem~\eqref{R}, and since $u_R$ is bounded we conclude that $u_R \leq v$ for all $R$. Thus, passing to the limit as $R \to +\infty$, we conclude the result.
\end{proof}

Now we investigate a more precise blow-up profile of the solution found in the previous proposition. For this, we require the following definition: we say that $u$ is a \textsl{strict viscosity solution} to problem~\eqref{eq} if
$$
\lambda u+ \Ls u + |Du|^m > f \quad \mbox{in} \ \Omega,
$$
in the viscositty sense.

\begin{lema}\label{perron-alternativo}
Assume hypotheses of Theorem~\ref{teo1} are in force. Suppose there exist $\bar U, \ubar U\in C(\Omega)\cap L^1(\Omega)$ with $\ubar U \leq \bar U $ in $\Omega$, with $\bar U$ supersolution of~\eqref{eq}, $\ubar U$ subsolution of~\eqref{eq} and such that one of them is strict. 

Then there exists a solution $u \in C^\alpha(\Omega) \cap L^1(\Omega)$ of \eqref{eq} satisfying $\ubar U\leq u\leq \bar U$. 
\end{lema}

The proof follows along the same lines of Proposition 2.2 in~\cite{DQTe} with straightforward adaptations to the censored case. 

\begin{lema}\label{perron}
Assume the hypotheses of Theorem~\ref{teo1} are in force. Then, for each $\epsilon \in (0,1)$, there exists a solution $v_\epsilon$ to~\eqref{eq} which satisfies
$$
C_0 - \epsilon \leq \liminf_{x \to \partial \Omega} \frac{v_\epsilon(x)}{g_{\gamma}(x)} \leq \limsup_{x \to \partial \Omega} \frac{v_\epsilon(x)}{g_{\gamma}(x)} \leq C_0 + \epsilon,
$$
\end{lema}

\begin{proof}
Recall definition of $C_0$  in~\eqref{defc0}. Consider the functions
$$
W_\epsilon^{\pm}(x) = (C_0 \pm \epsilon) g_{\gamma}(x) \pm \lambda^{-1} C_\epsilon, \ x \in \Omega.
$$

As in Proposition~\ref{propexistence}, we can take $C_\epsilon = C(\epsilon, f) > 0$ such that
$$
\lambda W_\epsilon^+ + \Ls W_\epsilon^+ + |DW_\epsilon^+|^m > f \quad \mbox{in} \ \Omega,
$$
from which $V_\epsilon^+$ is a strict supersolution to~\eqref{eq}. Similarly, $V_\epsilon^-$ is a strict subsolution. Thus by Lemma \ref{perron-alternativo} the exists  $v_\epsilon  \in C^\alpha(\Omega) \cap L^1(\Omega)$ a solution to \eqref{eq} such that $V_\epsilon^- \leq v_\epsilon \leq V_\epsilon$, from which the result follows.
\end{proof}

\begin{lema}\label{lema-Veron}
Assume the hypotheses of Theorem~\ref{teo1} are in force, and let $u$ be the minimal large solution of ~\eqref{eq} given by Proposition~\ref{propexistence}. Assume $f\geq 0$ and $v$ is a solution of~\eqref{eq} such that 
$$
v(x) \leq c_2 g_{\gamma}(x).
$$

Then, $u=v$ in $\Omega$.
\end{lema}

\begin{proof} By assumption on $v$ and the behavior of $u$ the exists a $\alpha \in (0,1)$ small enough such that $2 \alpha v\leq u$ in $\Omega$. Define
$$
w=u-\alpha(v-u).
$$

By the minimality of $u$, we know $u\leq v$ in $\Omega$ from which $w\leq u$ in $\Omega$. 

We claim that $w$ is a super solution to~\eqref{eq}. In fact, since $w=(1+\alpha)u-\alpha v$, using the convexity of $|\cdot|^m$ we have
$$
(1 + \alpha)|Du(x)|^m - \alpha|Dv|^m \leq |Dw|^m,
$$
from which we have
$$
\lambda w + \Ls w + |Dw|^m \geq (1 + \alpha) f - \alpha f = f,
$$
and $w$ is a supersolution for the problem. This is a formal proof that it is going to be justified by viscosity arguments at the end of the proof. Thus, accepting the claim


Now, taking $\theta \in (0, \alpha)$ we have $\theta u \leq w$ in $\Omega$. It is direct to see that $\theta u$ is a subsolution for the problem since
%
%
$$
\lambda (\theta u) + \Ls (\theta u) + |D(\theta u)|^m = |D(\theta u)|^m (1 - \theta^{1-m}) + \theta f 
\leq f\quad \mbox{in} \ \Omega,
$$
and the above computation can be justified in the viscosity sense.

Then, by Perron's method (see Lemma~\ref{perron-alternativo}) there exists $z$ solution to ~\eqref{eq}  such that $\theta u \leq z \leq w \leq u$. Then, if $v(x_0)>u(x_0)$ at some point $x_0\in \Omega$, then $z(x_0)\leq w(x_0)<u(x_0)$, which contradicts the minimality of $u$. This concludes the statement of the lemma.

\smallskip

Now we justify the claim on $w$. This seems to be part of the folclore of viscosity solutions theory but we provide the details next. Let $\varphi$ be a test function for $w$ at $\bar x_0 \in \Omega$, that is, $w(x_0) = \varphi(x_0)$ and $w \geq \varphi$ in $B_r(x_0)$ for some $r > 0$. We can assume $x_0$ is a strict minima for $w - \varphi$ in $B_r(x_0)$ and that $\bar B_{2r}(x_0) \subset \Omega$. Then, for $\epsilon > 0$ we double variables and consider
\begin{align*}
\Phi(x,y) = (1 + \alpha)u(x) - \alpha v(y) - \Big{(} (1 + \alpha) \varphi(x) - \alpha \varphi(x) - \epsilon^{-2}|x - y|^2 \Big{)},
\end{align*}
which attains it minimum at some $(\bar x, \bar y) \in \bar B_r(x_0)^2$. In particular, we have $\Phi(x_0, x_0) \geq \Phi(\bar x, \bar y)$ from which
$$
0 = w(x_0) - \varphi(x_0) \geq (1 + \alpha) (u(\bar x) - \varphi(\bar x)) - \alpha (v(\bar y) - \varphi(\bar y)) + \epsilon^{-2}|\bar x - \bar y|^2,
$$
from which $\epsilon^{-2}|\bar x - \bar y|^2$ remains bounded as $\epsilon \to 0$ (depending on the distance of $B_r(x_0)$ and $\partial \Omega$ since $u$ and $v$ are unbounded). Now, by standard viscosity arguments, we have that
$$
\bar x, \bar y \to x_0; \ u(\bar x) \to u(x_0), v(\bar y) \to v(x_0); \ \epsilon^{-2}|\bar x - \bar y|^2 \to 0,
$$
as $\epsilon \to 0.$ 

Then, from the inequality $\Phi(\bar x, \bar y) \leq \Phi(\bar y, \bar y)$, we see that $\bar x$ is a local minimum point to
$$
x \mapsto u(x) - (\varphi(x) - \frac{1}{\epsilon^2(1 + \alpha)}|x - \bar y|^2),
$$
and similarly, $\bar y$ is a local maximum point to
$$
y \mapsto v(y) - (\varphi(y) + \frac{1}{\epsilon^2 \alpha} |\bar x -y|^2).
$$

Denote 
$$
\varphi_1(x) = \varphi(x) - \frac{1}{\epsilon^2(1 + \alpha)}|x - \bar y|^2; \ \varphi_2(y) = \varphi(y) + \frac{1}{\epsilon^2 \alpha} |\bar x -y|^2; \ \xi = \frac{2}{\epsilon^2}(\bar x - \bar y).
$$

Using the viscosity inequalities for $u$ at $\bar x$ and for $v$ at $\bar y$, for all $\delta \in (0,r)$ we can write
\begin{align*}
\lambda u(\bar x) + \Ls[B_\delta(\bar x)] \varphi_1(\bar x) + \Ls[B_\delta(\bar x)^c] u(\bar x) + |D\phi(\bar x) - \frac{1}{1 + \alpha} \xi|^m \geq f(\bar x), \\
\lambda v(\bar y) + \Ls[B_\delta(\bar y)] \varphi_2(\bar y) + \Ls[B_\delta(\bar y)^c] v(\bar y) + |D\phi(\bar y) - \frac{1}{\alpha} \xi|^m \leq f(\bar y).
\end{align*}

Then, multiplying the first inequality by $1 + \alpha$, the second by $\alpha$ and subtracting the resulting inequalities, we use the convexity of $|\cdot|^m$  as above to arrive at
\begin{align}\label{testingw}
\lambda (w(x_0) - o_\epsilon(1)) + I^\delta + I_\delta + |(1 + \alpha) D\phi(\bar x) - \alpha D\phi(\bar y)|^m \geq (1 + \alpha) f(\bar x) - \alpha f(\bar y),
\end{align}
where $o_\epsilon(1) \to 0$ as $\epsilon \to 0$, and where for $\delta > 0$ we have denoted
\begin{align*}
 I_\delta = (1 + \alpha) \Ls[B_\delta(\bar x)] \varphi_1(\bar x) - \alpha \Ls[B_\delta(\bar y)] \varphi_2(\bar y) \\
 I^\delta = (1 + \alpha) \Ls[B_\delta(\bar x)^c] u(\bar x) - \alpha \Ls[B_\delta(\bar y)^c] v(\bar y).
\end{align*}

Now, using that $\Phi(\bar x + z, \bar y + z) \geq \Phi(\bar x, \bar y)$, for all $\delta \in (0, r)$ we see that
\begin{align*}
I^\delta \geq I^r + \int_{B_r \setminus B_\delta} [(1 + \alpha) \varphi(\bar x + z) - \alpha \varphi(\bar y + z) - ((1 + \alpha)\varphi(\bar x)) - \alpha \varphi(\bar y)] |z|^{-(N + 2s)}dz.
\end{align*}

Thus, since $\varphi_1, \varphi_2$ are smooth functions, we conclude that
\begin{align*}
I_\delta + I^\delta \geq & I^r + \mathrm{PV} \int_{B_r} [(1 + \alpha) \varphi(\bar x + z) - \alpha \varphi(\bar y + z) - ((1 + \alpha)\varphi(\bar x)) - \alpha \varphi(\bar y)] |z|^{-(N + 2s)}dz \\
& - \epsilon^{-2} o_\delta(1),
\end{align*}
where $o_\delta(1) \to 0$ as $\delta \to 0$. 

For such fixed $r \in (0,1)$, we replace this inequality into~\eqref{testingw}, and take limits as $\delta \to 0$ first and $\epsilon \to 0$ later to conclude, by using Dominated Convergence Theorem to control the integral terms, together with the smoothness of $\phi$, that
\begin{equation*}
\lambda w(x_0) + \Ls[B_r(x_0)] \varphi(x_0) + \Ls [B_r^c(x_0)]w(x_0) + |D\varphi(x_0)|^m \geq f(x_0),
\end{equation*}
from which we conclude the claim. The proof is now complete.
\end{proof}

Now we are in a position to provide the 
\begin{proof}[Proof of Theorem~\ref{teo1}]
By Proposition~\ref{propexistence} we have the minimal large solution to~\eqref{eq}. Using Lemmas~\ref{perron} and by the  the uniqueness in the $\gamma$-class given by Lemma ~\ref{lema-Veron}, we conclude that for each $\epsilon \in (0,1)$ we have
\begin{equation*}
C_0 - \epsilon \leq \liminf_{x \to \partial \Omega} \frac{u(x)}{g_{\gamma}(x)} \leq \limsup_{x \to \partial \Omega} \frac{u(x)}{g_{\gamma}(x)} \leq C_0 + \epsilon,
\end{equation*}
from which, taking $\epsilon \to 0$ we conclude~\eqref{rate-frac}. The estimates for the derivatives are a consequence of  Proposition ~\ref{lemaequiLip}.
\end{proof}

\begin{remark}
Although we do not know that the large solution is unique, by stability of viscosity solutions together with the boundary estimates, it is easy to see that this minimal solution converges to the unique large solution to~\eqref{eqLL} as $s \to 1$.
\end{remark}

\subsection{On the regime $s \leq 1/2$.} Motivated by the properties of the minimal solution found in Theorem~\ref{teo1}, we adopt the following terminology: we say that $u \in C(\Omega)$ is a \textsl{minimal large solution} to problem~\eqref{eq} if $u$ solves~\eqref{eq} in the viscosity sense, satisfies $u(x) \to +\infty$ as $d(x) \to 0^+$, and for every large solution $v$ to~\eqref{eq} we have $u \leq v$ in $\Omega$. By definition, minimal large solution is unique. Moreover, it is possible to prove that if $v$ is any large supersolution to~\eqref{eq}, then $u \leq v$ in $\Omega$.

When $s \leq 1/2$, then the growth constraint on the gradient means $m \leq 2s \leq 1$, being the case $s=1/2$ of particular interest since it is the limiting equation of the above studied. In this case, we have the following
\begin{prop}
Let $\Omega$ be a bounded domain with smooth boundary, $\lambda > 0$ and $f \in C_b(\Omega)$. If $s \leq 1/2$ and $m \leq 2s$, there is no minimal large solution for the problem
\begin{equation*}
\lambda u + \Ls u + |Du|^m = f \quad \mbox{in} \ \Omega.
\end{equation*}
\end{prop}

\begin{proof}
For $k, M > 0$, we consider the function $z = z_{k, M}$ given by
$$
z(x) = \lambda^{-1} M - k \log(d(x)), \quad x \in \Omega,
$$
where we have extended $d$ as a smooth, positive function to the whole of $\Omega$. From Lemma~\ref{lemabarrera0}, for all $x$ with $d(x) \leq \delta_0$ we have that
\begin{equation*}
\lambda z(x) + \Ls z(x) + |Dz(x)|^m \geq M + C(k d^{-2s}(x) + k^m d^{-m}(x)) \geq M,
\end{equation*}
with $C > 0$. For $d(x) > \delta_0$, by the smoothness of $d$ and the integrability of $\log(d(x))$ we have the existence of $C_{\delta_0}$ such that
\begin{equation*}
\lambda z(x) + \Ls z(x) + |Dz(x)|^m \geq M - k C_{\delta_0}.
\end{equation*}

Thus, we fix $M > 0$ large enough in terms of $\delta_0, f$ and the data in order to get $z$ is a blow-up supersolution to~\eqref{eq} for all $k \in (0,1)$. Hence, if there is a minimal solution to~\eqref{eq}, we have $u \leq z_k$, and taking $k \to 0^+$ we conclude $u$ is bounded above by $M$, which is a contradiction. 
\end{proof}

%
%
%
%
%

\section{The ergodic problem}
\label{secerg}

From now on, we will assume the hypotheses of Theorem~\ref{teo2} are in force. 
The main result of this section is the following
\begin{teo}\label{teoergodic}
There exists $(u, c) \in C(\Omega) \times \R$ solving the problem~\eqref{eq-erg-frac}-\eqref{blow-up}. Moreover, we have that $u \in C^{1, \alpha}(\Omega)$ and it is in the $\gamma$-class~\eqref{gamma-class}.
\end{teo}

In order to obtain the above result, we follow the nowadays standard vanishing discount method mentioned in the Introduction, namely, take a suitable limit as $\lambda \to 0$ in~\eqref{eq}. 

We recall that $u_\lambda$ found in Theorem~\ref{teo1} satisfies
\begin{equation*}
0 \leq u_\lambda \leq \frac{C}{\lambda} +  (C_0+\varepsilon) g_{\gamma}\quad \mbox{in} \ \Omega.
\end{equation*}

Fix $x_0 \in \Omega$ and define 
\begin{equation}\label{defwlambda}
w_\lambda(x) = u_\lambda(x) - u_\lambda(x_0), \quad x \in \Omega,
\end{equation}
which satisfies the equation
\begin{equation}\label{eqvanishing}
\lambda w_\lambda - \Delta_\Omega w_\lambda + |Dw_\lambda|^m - f + \lambda u_\lambda(x_0) = 0 \quad \mbox{in} \ \Omega.
\end{equation}

Now we give some bounds for $w_\lambda$, here we use some ideas of \cite{BT}.

\begin{lema} \label{boundw}
Let $w_\lambda$ be as in~\eqref{defwlambda}. Then, there exists $C > 0$ and $0<\delta \leq \delta_0 /2$ such that, for all $\lambda \in (0,1)$ we have
\begin{equation*}
-\max_{\bar \Omega_\delta}|w_\lambda|-C\leq  w_\lambda \leq \max_{\bar \Omega_\delta}|w_\lambda|+Cg_\gamma \quad \mbox{in} \ \Omega.
\end{equation*}

Moreover, we have
\begin{equation}\label{unifwlambda}
\sup_{\lambda\in(0,1)}\max_{\bar \Omega_\delta}|w_\lambda|<\infty.
\end{equation}
\end{lema}

\begin{proof} Fix $\epsilon > 0$ and take $W = W_\epsilon = (C_0 + \epsilon) g_\gamma$.
By the boundary behavior of $u_\lambda$, there exists $x_\lambda \in \mbox{argmax}_{\Omega} \{ u_\lambda-W \}$. Recalling that $W$ is smooth, we can use $W$ as test function for $u_\lambda$ as subsolution for~\eqref{eq} at $x_\lambda$, from which
\begin{align*}
0 \geq & \Ls W(x_\lambda)+ |DW(x_\lambda)|^m - f(x_\lambda) + \lambda u_\lambda(x_\lambda) \\
\geq & \Ls W(x_\lambda)+ |DW(x_\lambda)|^m - f(x_\lambda) - \inf \{ f^- \}.
\end{align*}

Then, if $x_\lambda \in \Omega \setminus \Omega_{\delta_0}$, we have
\begin{equation*}
f(x_\lambda) + \inf \{ f^-\} \geq \tilde C_1 d^{\gamma - 2s}(x_\lambda),
\end{equation*}
for some $\tilde C_1 > C_1$. Then, by the boundary behavior of $f$ (see~\eqref{compor-f}), we conclude the existence of $\delta \in (0, \delta_0)$ such that $x_\lambda \in \bar \Omega_\delta$ for all $\lambda \in (0,1)$.
Thus, there exists $C>0$ not depending on $\lambda$ such that
$$
 w_\lambda(x) \geq W(x)-W(x_\lambda)+u_\lambda(x_\lambda)-u_\lambda(x_0) \geq -\max_{\bar \Omega_\delta}|w_\lambda|-C, \quad x \in \Omega.
$$

For the upper bound, we see that for $C > C_0 + \epsilon$, the function $V(x) = \max_{\bar \Omega_\delta} |w_\lambda| + Cg_\gamma, \ x \in \Omega$, is a supersolution to the equation solved by $w_\lambda$ in $\Omega \setminus \bar \Omega_\delta$ for $\delta$ small enough, and clearly $w_\lambda \leq V$ on $\bar \Omega_\delta$. By the boundary behavior of $w_\lambda$ we can apply comparison principle to conclude that $w_\lambda \leq V$ in $\Omega$ and the upper bound follows.

%

\smallskip

Finally, we will prove~\eqref{unifwlambda}. Assume by contradiction that there exists a sequence $\lambda_n \to 0$ 
such that $t_n := \max_{\bar \Omega_\delta} |w_{\lambda_n}|\to \infty$.

Denote 
$$
v_n(x)=\frac{w_{\lambda_n(x)}}{t_n} \quad\mbox{for } \ x\in \Omega.
$$

By the estimates on $w_\lambda$ we have
\begin{equation}\label{first-ineq}
-1 + o_n(1) \leq v_n \leq 1+ \frac{V(x)}{t_n}\quad \mbox{in} \ \Omega,
\end{equation}
where $o_n(1) \to 0$ as $n \to \infty$.

Since $v_n$ satisfies
$$
\lambda_n v_n(x) + \Ls v_n(x) + t_n^{m-1}|Dv_n(x)|^m = t_n^{-1}(f + \lambda_n u_{\lambda_n}(x_0)) \quad \mbox{in} \ \Omega,
$$
by the uniform $C^{0,\frac{1}{2}}_{loc}$ estimates of Lemma \ref{lemaLipbounded} we have the family is locally equicontinuous. Then, up to subsequence $v_n$ converges locally uniformly
to some continuous $v$ and by stability results 
$$|Dv|= 0\quad \mbox{in} \ \Omega,$$
Moreover $v$ satisfies
$$|v|\leq 1 \quad \mbox{in} \ \Omega; \quad v(x_0)=0\quad \mbox{and}\quad \max_{K_\delta}|v|=1  $$
Let prove that v is constant $\Omega$ so we find a contrtadiction. In fact, take any $x_1 \in \Omega$ and $\mu>0$ such that $B(x_1, \mu)\subset \Omega$.
Define in $B(x_1, \mu)$ the function
$$\phi_\varepsilon(x)=\frac{\varepsilon}{(\mu^2-|x_1-x|^2)}$$
Then $v-\phi_\varepsilon$ has a maximum $x_2$ therefore $\phi_\varepsilon$ is a test function at $x_2$ if $x_1\not =x_2$  
then $|D \phi_\varepsilon|\not =0$ a contradiction. So  $x_1=x_2$ and   then
 $v(x_1)-\phi_\varepsilon(x_1)\geq v(x)-\phi_\varepsilon(x)$ in  $B(x_1, \mu)$. Now we let $\varepsilon \to 0$ to find $v(x_1)\geq v(x)$ for all  $x\in B(x_1, \mu)$.
Since  $x_1 \in \Omega$ and $\mu$ and any such that  $B(x_1, \mu)\subset \Omega$ we get that $v$ is constant. This is a contradiction, and the proof is complete.
\end{proof}

Now we are position to prove our main Theorem

\begin{proof}[Proof Theorem \ref{teoergodic}] 
Form Lemma  \ref{boundw} we have $w_\lambda$ is locally bounded. By Lemma~\ref{lemaequiLip} we have $w_\lambda$ is locally equicontinuous so there exist a subsequence $\lambda_k\to 0, c \in \R$ and $u\in C(\Omega)$ such that  
\begin{equation}\label{uk}
\left \{ \begin{array}{rl} \lambda_k u_{\lambda_k}(x_0) & \to c, \\
w_{\lambda_k} & \to u \ \mbox{in $L^\infty_{loc}(\Omega)$,}
\end{array} \right .
\end{equation}
ans $k \to \infty$. By stability of the viscosity solutions, the pair $(u,c)$ solves \eqref{eq-erg-frac}.
\end{proof}

\section{Qualitative properties of the ergodic problem}
\label{secqual}

From here, $(u,c)$ is the solution found in Theorem~\ref{teoergodic}. Notice that $(u+C, c)$ with $C$ a constant is also a solution to~\eqref{eq-erg-frac}, but in the analysis below this arbitrary constant plays no role.
\subsection{Characterization of the ergodic constant.}
Assume $f \geq 0$. For each $R, \lambda > 0$, there exists unique function $u_{\lambda, R}$ solving~\eqref{R}, that is
\begin{equation} \label{RR}
\begin{cases}
\lambda u_{\lambda, R} + \Ls u_{\lambda, R} + |Du_{\lambda, R}|^m  = \min \{ f, R \} \quad \text{in }\Omega\\
u_{\lambda, R} =R\quad \text{in }\partial\Omega
\end{cases}
\end{equation}
and as we saw in the proof of Proposition~\ref{propexistence}, $u_{\lambda, R} \nearrow u_\lambda$ as $R \nearrow +\infty$ locally uniformly in $\Omega$, where $u_\lambda$ given in Theorem~\ref{teo1}.


%

In this section we will prove the following theorem
\begin{prop}\label{char}
Let $(u,c)$ be a solution of \eqref{eq-erg-frac}, then 
\begin{align*}
c=\inf\{\rho: \exists u \in C(\Omega)\cap L^1_\omega(\Omega) \text{ satisfying }\eqref{blow-up} \text{ and } \Ls u + |Du|^m - f+\rho \geq 0 \}
\end{align*}
\end{prop}

The proof of Proposition~\eqref{char} follows the steps done in \cite{BT}. 
\begin{proof}
Let us denote 
\[
c^*=\inf\{\rho: \exists u \in C(\Omega)\cap L^1_\omega(\Omega) \text{ satisfying }\eqref{blow-up} \text{ and } \Ls u + |Du|^m - f +\rho\geq 0 \}
\]
and note that by definition $c^*\leq c$. 

Suppose now by contradiction that $c^*<c$. In this case, by the definition of infimum there is $\rho\in(c^*,c)$ and a function $w\in C(\Omega)\cap L^1_\omega(\Omega)$ solving
\[
\Ls w + |Dw|^m - f +\rho\geq 0\quad \text{in }\Omega, 
\]
satisfying the blow-up condition \eqref{blow-up}, furthermore we can assume $w\geq0$. Now let $(u,c)$ be a solution of \eqref{eq-erg-frac} and denote $u_k = u_{\lambda_k}, w_k = w_{\lambda_k}$ 
and $\lambda_k \to 0$ as in~\eqref{uk}.

Given $M>0$, there is $k_0\in\N$ such that 
\begin{align*}\label{rho}
\rho<\lambda_k u_{k}(x_0)-\lambda_k M, 
\end{align*}
for all $k\geq k_0$. Furthermore, as we mentioned in the beggining of this section, for each $k$ we have
\[
u_{\lambda_k, R}\to u_{k}\quad \text{as }R\to\infty
\]
in compact subsets $\Omega$. Then, there exists $R_k > 0$ such that, for all $R \geq R_k$ we have
\[
\rho<\lambda_k u_{\lambda_k, R}(x_0)-\lambda_k M.
\]

For such $R$, denote $U_R(x)=u_{\lambda_k, R}(x)-u_{\lambda_k, R}(x_0)+M$, and observe that 
\begin{equation}\label{eqU}
\lambda_k U_R + \Ls U_R +|DU_R|^m - f =-\lambda_k u_{\lambda_k, R}(x_0)+\lambda_k M\quad\text{in }\Omega,
\end{equation}
with $U_R =R-u_{\lambda_k, R}(x_0)+M\leq R+M$ on $\partial\Omega$.
Now, from the above estimates we see that $w$ satisfies
\begin{align*}
\lambda_k w + \Ls w +|Dw|^m - f& \geq - \rho + \lambda_k w \\ & \geq -\lambda_k u_{\lambda_k, R}(0)+\lambda_k M,
\end{align*}
that is $w$ is a super-solution for~\eqref{eqU}. By comparison principle, we deduce then $w\geq U_R$ in $\Omega$. By letting $R\to \infty$ and then $k\to\infty $ ($\lambda_k\to 0$), we conclude that
\[
w\geq u+M\quad\text{in }\Omega.
\]

The previous inequality holds for any arbitrary $M>0$, which would imply that $w$ is unbounded in $\Omega$, which is a contradiction.
\end{proof}


\subsection{On uniqueness.} In what follows we present some partial uniqueness results for the ergodic problem. The first results establishes that the pair $(u,c)$ found in Theorem~\ref{teoergodic} is such that $u$ is ``minimal" (up to a constant).

\begin{prop}\label{propuniqueness}
Let $(u, c)$ be a solution to~\eqref{eq-erg-frac} given in Theorem~\ref{teoergodic}. Then, for every $(v,c') \in C(\Omega) \times \R$ solution to~\eqref{eq-erg-frac}-~\eqref{blow-up},
there exists $M > 0$ such that $u - M \leq v$ in $\Omega$.
\end{prop}

\begin{proof}
Denote $u_k = u_{\lambda_k}$ and $w_k = w_{\lambda_k}$ as in~\eqref{uk}. 
Notice that $u(x_0) = 0$ by construction, and since we assume $f$ is H\"older, $\{ w_k \}_k$ is uniformly $C^{2s + \alpha}_{loc}(\Omega)$.


Fix $K \subset \subset \Omega$ with positive measure. Since $v$ is bounded from below, for each $M > 0$, there exists $C = C(M, u, v, K) > 0$ such that, the function $\tilde u = u -C$ satisfies 
$$
\tilde u + M + 2 < v \quad \mbox{in} \ K.
$$

From this point, we take $M$ large enough in order the inequality
\begin{align}\label{choiceM}
M C_{N, s} |K| \mathrm{dist}(K, \partial \Omega)^{-(N + 2s)} > |c'| + |c| + 1
\end{align}
holds. We claim that for this $M$, we conclude that $\tilde u \leq v$, from which the result follows.

By contradiction, we assume there exists $\tilde x \in \Omega \setminus K$ such that $\tilde u(\tilde x) > v(\tilde x)$.  We consider $\tilde K \subset \subset \Omega$ such that $\{ \tilde x \} \cup K \subset \subset \tilde K$. For such a set $\tilde K$, there exists $R_k$ large enough such that the function $\tilde w_k(x) = u_{\lambda_k, R}(x) - u_{\lambda_k, R}(x_0)$ satisfies
$$
\| \tilde w_k - u \|_{L^\infty(\tilde K)} \leq k^{-1},
$$
for all $R \geq R_k$. Here, $u_{\lambda, R}$ solves~\eqref{RR}.

Then, for all $k$ large enough we have that
$$
(\tilde w_k - C) + M + 1 \leq v \quad \mbox{in} \ K,
$$
and there exists $\tilde x_k \in \Omega \setminus K$ such that $\tilde w_k(\tilde x_k) - C > v(\tilde x_k)$. Since $\tilde w_k$ is bounded in $\Omega$ and $v$ blows-up on the boundary, we have the existence of $\bar x_k \in \Omega \setminus K$ such that
$$
(\tilde w_k(\bar x_k) - C) - v(\bar x_k) = \max_{\Omega} \{ (\tilde w_k - C) - v \} > 0.
$$

Then, we can use $\tilde w_k \in C^{2s + \alpha}_{loc}(\Omega)$ as a test function for $v$ at $\bar x_k$, from which, for all $\delta > 0$ small enough such that $B_\delta(\bar x_k) \subset \Omega$ we have
\begin{align*}
0 \leq (-\Delta)_\Omega^s[B_\delta(\bar x_k)] \tilde w_k(\bar x_k) + (-\Delta)_\Omega^s[B_\delta(\bar x_k)^c] v(\bar x_k) + |D \tilde w_k(\bar x_k)|^m - f(\bar x_k) + c'.
\end{align*}

Taking $\delta$ small we can assume $B_\delta(\bar x_k) \subset \Omega \setminus K$, and from this we see that
\begin{align*}
(-\Delta)_\Omega^s[B_\delta(\bar x_k)^c] v(\bar x_k) = & -C_{N,s} \int_{K} \frac{v(z) - v(\bar x_k)}{|\bar x_k - z|^{N + 2s}}dz -C_{N, s} \int_{(\Omega \setminus K) \setminus B_\delta(\bar x_k)} \frac{v(z) - v(\bar x_k)}{|\bar x_k - z|^{N + 2s}}dz \\
\leq & -C_{N,s} \int_{K} \frac{\tilde w_k(z) + M + 1 - \tilde w_k(\bar x_k)}{|\bar x_k - z|^{N + 2s}}dz \\
&  -C_{N, s} \int_{(\Omega \setminus K) \setminus B_\delta(\bar x_k)} \frac{\tilde w_k(z) - \tilde w_k(\bar x_k)}{|\bar x_k - z|^{N + 2s}}dz.
\end{align*}

Using this into the viscosity inequality for $v$ at $\bar x_k$ together with the equation satisfied by $\tilde w_k$ we arrive at
\begin{align*}
0 \leq & -C_{N, s} (M + 1) \int_{K}\frac{dz}{|\bar x_k - z|^{N + 2s}} + (-\Delta)^{s} \tilde w_k(\bar x_k) + |D \tilde w_k(\bar x_k)|^m - f(\bar x_k) + c' \\
\leq & -M C_{N,s} |K| \mathrm{dist}(K, \partial \Omega)^{-(N + 2s)} + c' - \lambda_k u_{\lambda_k, R}(\bar x_k) - \lambda_k u_{\lambda_k, R}(x_0) \\
\leq & -M C_{N,s} |K| \mathrm{dist}(K, \partial \Omega)^{-(N + 2s)} + c' - \lambda_k u_{\lambda_k, R}(x_0),
\end{align*}
since we can assume that $u_{\lambda, R} \geq 0$ for all $\lambda$ and $R$. Now, since $u_{\lambda, R} \nearrow u_\lambda$ locally uniform in $\Omega$ as $R \nearrow +\infty$, by takin this limit we get
$$
0 \leq  -M C_{N,s} |K| \mathrm{dist}(K, \partial \Omega)^{-(N + 2s)} + c' - \lambda_k u_{\lambda_k}(x_0),
$$
and by taking $k$ large enough we arrive at
\begin{equation*}
M C_{N,s} |K| \mathrm{dist}(K, \partial \Omega)^{-(N + 2s)} \leq c' + c + 1/2,
\end{equation*}
which is a contradiction with the choice of $M$ in~\eqref{choiceM}. This concludes the proof.
\end{proof}

\begin{lema}\label{lemma1} 
Let $(u, c)$ be the ergodic pair found in Theorem~\ref{teoergodic}. Then, for every  supersolution $(v,c)$ to the ergodic problem~\eqref{eq-erg-frac}-\eqref{blow-up}, there exists $C \in \R$ such that $u = v + C$.
\end{lema}

\begin{proof}
By the previous result, there exists $C > 0$ such that $v + C \geq u$ in $\Omega$. Let $C^* = \inf \{ C : v + C \geq u \mbox{ in $\Omega$} \}$. The infimum is attained, we have $v + C^* \geq u$ in $\Omega$ and moreover
\begin{equation}\label{inf=0}
\inf_\Omega \{ v + C^* - u \} = 0.
\end{equation}

We claim that $v + C^* = u$. By contradiction, we assume then the open set 
$$
\OO = \{ x \in \Omega : v(x) + C^* > u(x) \}
$$
 is nonempty. 


Let $\lambda_k, u_k = u_{\lambda_k}$ as in~\eqref{uk}, and with this, let $u_{k, R} = u_{\lambda_k, R}$ as in~\eqref{RR}. Then, by the local, monotone convergence of this family, there exists $\theta > 0$ and $R_k$ such that, for all $R \geq R_k$ we have $\tilde w_{k, R} = u_{\lambda_k, R} - u_{\lambda_k, R}(x_0)$ is such that 
$$
v + C^* - \tilde w_{k, R} > \theta \quad \mbox{in} \ K \cap \OO,
$$
for all $R \geq R_k$, where $K \subset \subset \Omega$ is compact with nonempty interior such that $K \cap \OO$ contains an open set. Since $v$ satisfies~\eqref{blow-up}, we have there exists $\tilde x_k \in \Omega$ such that $\inf_\Omega \{ v + C^* - \tilde w_k \} = (v + C^* - \tilde w_{k, R})(\tilde x_k)$ for all $R \geq R_k$. From~\eqref{inf=0}, taking $k$ and $R_k$ large enough, we may assume that $(v + C^* - \tilde w_{k, R})(\tilde x_k) < \theta/2$.


Then, we can test the equation for $v$ at $\tilde x_k$ with test function $\tilde w_k$, from which we get
\begin{align*}
0 \leq (-\Delta)_\Omega^s[B_\delta(\tilde x_k)] \tilde w_k(\tilde x_k) + (-\Delta)_\Omega^s[B_\delta(\tilde x_k)^c] v(\tilde x_k) + |D \tilde w_k(\tilde x_k)|^m- f(\tilde x_k) + c.
\end{align*}

Denoting $E = K \cap \OO$, we can take $\delta$ small enough in order to $B_\delta(\tilde x_k) \cap (K \cap \OO) = \emptyset$, from which
\begin{align*}
(-\Delta)_\Omega^s[B_\delta(\tilde x_k)^c] v(\tilde x_k) = & -C_{N,s} \int_{E} \frac{v(z) - v(\tilde x_k)}{|\tilde x_k - z|^{N + 2s}}dz -C_{N, s} \int_{\Omega \setminus (E \cup B_\delta(\tilde x))} \frac{v(z) - v(\tilde x_k)}{|\tilde x_k - z|^{N + 2s}}dz \\
\leq & -C_{N,s} \theta \int_{E} \frac{dz}{|\tilde x_k - z|^{N + 2s}} + \Ls [\Omega \setminus B_\delta(\tilde x)] \tilde w_k(\tilde x_k) \\
\leq & -C_{N,s} \theta |E| \mathrm{diam}(\Omega)^{-(N + 2s)} + \Ls [\Omega \setminus B_\delta(\tilde x)] \tilde w_k(\tilde x_k),
\end{align*}
where $|E| > 0$ is the Lebesgue measure of $E$. Replacing this estimate on the viscosity inequality, we get
\begin{align*}
0 \leq & -C_{N,s} \theta |E| \mathrm{diam}(\Omega)^{-(N + 2s)} + \Ls \tilde w_k(\tilde x_k) + |D \tilde w_k(\tilde x_k)|^m- f(\tilde x_k) + c \\
= & -C_{N,s} \theta |E| \mathrm{diam}(\Omega)^{-(N + 2s)} -\lambda_k \tilde w_{\lambda_k, R}(\tilde x_k) - \lambda_k u_{\lambda_k, R}(x_0) + c.
\end{align*}

Now, since $\tilde w_k(\tilde x_k, R) > v(\tilde x_k) + C^* - \theta/2 > -C$ for some $C$ not depending on $k$ nor $R$, we take $R \to \infty$ and then $k \to \infty$ to conclude, by~\eqref{uk}, that
$$
0 \leq -C_{N,s} \theta |E| \mathrm{diam}(\Omega)^{-(N + 2s)}, 
$$
which is a contradiction.
\end{proof}

\begin{lema}\label{lemma2} 
Let $(u, c)$ be the ergodic pair found in Theorem~\ref{teoergodic}, and let $(v,c')$ be an ergodic pair~\eqref{eq-erg-frac}-\eqref{blow-up} such that $v$ is in the $\gamma$-class~\eqref{gamma-class}. Then, $c = c'$.
\end{lema}

\begin{proof}
We know that $c \leq c'$, and we can assume that $v \geq u$ in $\Omega$ by Proposition~\eqref{propuniqueness}.
Denoting $\tilde f = f - c + \lambda v$, we have $\tilde f$ satisfies~\eqref{compor-f} with the same constant $C_1$, and $v$ solves
\begin{align*}
\lambda v + \Ls v + |Dv|^m = f - c + \lambda v = \tilde f \quad \mbox{in} \ \Omega.
\end{align*}

Then, $v$ is the minimal solution to this problem and therefore it satisfies~\eqref{rate-frac}.
Similar argument can be done on $u$ from which the same boundary blow-up rate holds. Then, for each $\mu \in (0,1)$ we have $u > \mu v$ near the boundary. Morever, there exists $x = x_\mu \in \Omega$ such that $\inf_\Omega \{ u - \mu v \} = u(x) - \mu v(x)$. Then, we have
\begin{align*}
0 \leq & \mu \Ls v(x) + \mu^m |Dv(x)|^m - f(x) + c \\
\leq & \mu(\Ls v(x) + |Dv(x)|^m) - f(x) + c \\
\leq & \mu(f(x) - c') - f(x) + c, 
\end{align*}
from which $\mu c' \leq (\mu - 1)f(x_\mu) + c$. Since $f$ is bounded from below independent of $\mu$, we have $\mu c' \leq C(1 - \mu) + c$ for some $C > 0$ not depending on $\mu$, and taking limit as $\mu \to 1$ we arrive at $c' \leq c$. This concludes the proof.
%
\end{proof}
Finally, we are in a position to
\begin{proof}[Proof Theorem \ref{teo2}] 
Existence of the ergodic pair $(u,c)$ is a consequence of Theorem~\ref{teoergodic}. The characterization of the ergodic constant $c$ as a critical value is given in Proposition~\ref{char}. Minimality of the ergodic solution $u$ is proven in Proposition~\ref{propuniqueness}. Finally, the partial uniqueness results can be found in Lemmas~\ref{lemma1}, and~\ref{lemma2}. The estimate ~\eqref{rate-frac} holds 
since $u$ solves the problem
$$
\lambda u + \Ls u + |Du|^m = f - c + \lambda u \quad \mbox{in} \ \Omega,
$$
from which, we apply Theorem~\ref{teo1} with $C_1$ replaced by $C_1 + C\lambda$ for some $C > 0$ not depending on $\lambda$. Since $C_0$ defined in~\eqref{defc0} is continuous with respect to $C_1$, taking $\lambda \searrow 0$ we conclude the result. The estimates for the derivatives are a consequence of Proposition~\ref{lemaequiLip}.
\end{proof}



\noindent
{\bf Acknowledgements.} Part of this work was developed in the context of research visits at Departamento de Matemáticas of Universidad Técnica Federico Santa María, Valparaiso, Chile; and Instituto de Matemáticas, Universidade Federal do Rio de Janeiro, RJ, Brazil. The authors would like to thank the hospitality of both institutions.

A.Q. was partially supported by Fondecyt Grant 1231585. 
E.T. was partially supported by CNPq Grant 306022/2023-0 and FAPERJ APQ1 Grant 210.573/2024. 
Both authors were supported by CNPq Grant 408169/2023-0.



\begin{thebibliography}{00}

\bibitem{Ab} 
Abatangelo, N. {\em Large $s$-harmonic functions and boundary blow-up solutions for the fractional laplacian.} Discrete Contin. Dyn. Syst. 35 (2015), no. 12, 5555-5607.


\bibitem{A} Arapostathis, A., Biswas, A., Caffarelli, L. {\em On uniqueness of solutions to viscous HJB equations with a subquadratic nonlinearity in the gradient.} Communications in Partial Differential Equations 44 (2019), 1466–1480.

\bibitem{reg3} Audrito, A., Felipe-Navarro, J.C., and Ros-Oton, X. {\em The Neumann problem for the fractional Laplacian: regularity up to the boundary.} Ann. Scuola Norm. Sup. - Classe di Scienze, 24(2), 1155-1222 (2023). 

\bibitem{BCCI}
Barles, G., Chasseigne, E., Ciomaga, A. and Imbert, C. {\em Lipschitz Regularity of Solutions for Mixed Integro-Differential Equations.}
J. Diff. Eq., 252 (2012), 6012-6060.

\bibitem{BCCI14}
Barles, G., Chasseigne, E., Ciomaga, A., and Imbert, C. {\em Large time behavior of periodic viscosity solutions for uniformly parabolic integro-differential equations.}
Calc. Var. Partial Differential Equations 50 (2014), no. 1-2, 283-304.

\bibitem{bcgj} 
Barles, G., Chasseigne, E., Georgelin, C., and Jakobsen, E.
{\em On Neumann type problems for nonlocal equations set in a half space.}
 Trans. Amer. Math. Soc. 366 (2014), no. 9, 4873–4917.

\bibitem{BLT17}
Barles, G., Ley, O., and Topp, E. {\em Lipschitz regularity for integro-differential equations with coercive Hamiltonians and application to large time behavior.} Nonlinearity, 30(2):703-734, 2017.


\bibitem{BKLT15}
Barles, G., Koike, S., Ley, O., and Topp, E. {\em Regularity results and large time behavior for integro-differential equations with coercive Hamiltonians.} Calc. Var. Partial Differential Equations, 54(1):539-572, 2015.

\bibitem{BT16}
Barles, G., and Topp, E.
\newblock {\em Lipschitz regularity for censored subdiffusive integro-differential
  equations with superfractional gradient terms.}
\newblock {\em Nonlinear Anal.}, 131:3--31, 2016.


\bibitem{B2} Barles, G., and Meireles, J. {\em On unbounded solutions of ergodic problems in Rm for viscous
Hamilton–Jacobi equations.} Communications in Partial Differential Equations 41 (2016),
1985–2003.

\bibitem{B3} Barles, G., Porretta, A., Tchamba, T.T. {\em On the large time behavior of solutions
of the Dirichlet problem for subquadratic viscous Hamilton–Jacobi equations.} Journal de
Math pures et Appl. 94 (2010), 497–519.




\bibitem{BDGQ} Barrios, B., Del Pezzo, L.M., Garcia-Melian, J., and Quaas, A. {\em A priori bounds and existence of solutions for some nonlocal elliptic problems.}
Revista Matemática Iberoamericana 34 (1), 195-220

\bibitem{BDL19}
Birindelli, I., Demengel F. and Leoni, F. {\em Ergodic pairs for singular or degenerate fully nonlinear
operators.} ESAIM Control Optim. Calc. Var., 25 (2019), Paper No. 75, 28 pp.

\bibitem{BDL21}
Birindelli, I., Demengel, F. and  Leoni, F. {\em Boundary asymptotics of the ergodic functions associated with fully nonlinear operators through a Liouville type theorem.}
Discrete Contin. Dyn. Syst. 41 (2021), no. 7, 3021-3029.

\bibitem{BQT}
Biswas, A., Quaas, A., and Topp, E. {\em Nonlocal Liouville theorems with gradient nonlinearity.} J. Funct. Anal., 289, no. 8, Paper No. 111008, 44 pp. 2025


\bibitem{BT}
Biswas, A., Topp, E. {\em Nonlocal ergodic control problem in $\R^d$.} Math. Annalen, 390, 45--94, 2024

 
 \bibitem{BBC}
Bogdan, K., Burdzy, K. and Chen Z.-Q. {\em Censored Stable Processes.} Prob. Theory and Rel. Fields, Vol. 127, Issue 1 (2003), pp 89-152.

\bibitem{BL21}
Buccheri, S., and  Leonori, T. {\em Large solutions to quasilinear problems involving the  $p$-Laplacian as $p$ diverges} Calc. Var. Partial Differential Equations 60 (2021), no. 1, Paper No. 30, 23 pp.

\bibitem{CS09}
Caffarelli, L. and Silvestre, L.
\newblock {\em Regularity theory for fully nonlinear integro-differential equations.}
Comm. Pure Appl. Math., 62(5):597--638, 2009

\bibitem{CFQ}
Chen, H.; Felmer, P. and Quaas, A. {\em Large solutions to elliptic equations involving fractional Laplacian.} Ann. Inst. Henri Poincaré, Analyse non lineáire 32 (2015), 1199-1228.

\bibitem{CH}
Chen, H; Hajaiej, H. {\em Boundary blow-up solutions of elliptic equations involving regional fractional Laplacian.} Comm. Math. Sci. (2019) Vol. 17, No. 4, pp. 989-1004.

\bibitem{DQT17} Dávila, G., Quaas A., and Topp, E. {\em Continuous viscosity solutions for nonlocal Dirichlet problems with coercive gradient terms,} Math. Annalen 369 (3-4), 1211-1236, 2017.

\bibitem{DQTe}
D\'avila, G., Quaas, A., and Topp, E. {\em On large solutions for fractional Hamilton-Jacobi equations.} Proceedings of the Royal Society of Edinburgh Section A: Mathematics. Vol. 154, Issue 5 (2024) 1313-1335.

\bibitem{DQThar}
Dávila, G., Quaas, A., and Topp, E. {\em Large harmonic functions for fully nonlinear fractional operators.} Comm. Partial Differential Equations 49 (2024), no. 10-12, 919-937.



\bibitem{reg1} Fall, M. M., and Ros-Oton, X. {\em Global Schauder theory for minimizers of the $H^s(\Omega)$ energy.} Journal of Functional Analysis, 283,  3, 2022.

\bibitem{reg4} Fall, M. M. {\em Regional fractional Laplacians: boundary regularity.} Journal of Differential Equations, 320, 25  2022.

\bibitem{FGMP13}
Ferone, V., Giarrusso, E., Messano, B., and Posteraro, M. R. {\em Isoperimetric inequalities for an ergodic stochastic control problem.} Calc. Var. Partial Differential Equations 46 (2013), no. 3-4, 749-768.

\bibitem{reg2} Frank, R.L., Jin, T., and Wang, W. {\em On the sharp constants in the regional fractional Sobolev inequalities.} Partial Differential Equations and Applications, 6,15  ,2025.
 

\bibitem{G17}
Ghilli, D.
\newblock {\em On {N}eumann problems for nonlocal {H}amilton-{J}acobi equations with
  dominating gradient terms.}
\newblock {\em Calc. Var. Partial Differential Equations}, 56(5):Paper No. 139,
  41, 2017.


\bibitem{grubb1} Grubb, G. {\em Resolvents for fractional-order operators with nonhomogeneous local boundary conditions.} J. Funct. Anal. 284 (2023), no. 7, No. 109815, 55 pp.. 

\bibitem{IN}
Ishii, H. and Nakamura, G. {\em A Class of Integral Equations and Approximation of p-Laplace Equations.} Calc. Var. PDE (2010), no. 37, 485--522. 


\bibitem{LL}
Lasry, J.M. and Lions, P.L. {\em Nonlinear elliptic Equations with Singular Boundary Conditions and Stochastic 
Control with State Constraints.} Math. Ann. 283, 583-630 (1989).

\bibitem{LPV}
Papanicolaou~G., Lions, P.L. and S.R.S. Varadhan.
\newblock Homogeneization of hamilton-Jacobi equations.
\newblock {\em Unpublished}, 1986.



\bibitem{S15}
Serra, J.
\newblock {\em {$C^{\sigma+\alpha}$} regularity for concave nonlocal fully nonlinear
  elliptic equations with rough kernels.}
\newblock {\em Calc. Var. Partial Differential Equations}, 54(4):3571--3601,
  2015.


\bibitem{RT}
Rossi, J.D., and Topp, E. {\em Large solutions for a class of semilinear integro-differential equations with censored jumps.} J. Differential Equations 260 (2016) 6872-6899.



\bibitem{T} Tchamba, T.T. {\em Large time behavior of solutions of viscous Hamilton-Jacobi equations with superquadratic Hamiltonian.} Asymptot. Anal., 66(3-4):161–186, 2010.


\end{thebibliography}
\end{document}